\documentclass[a4paper, 10pt, twoside, notitlepage]{amsart}

\usepackage{amsmath,amscd}
\usepackage{amssymb}
\usepackage{amsthm}
\usepackage{comment}
\usepackage{graphicx, xcolor}

\usepackage{mathrsfs}
\usepackage[ocgcolorlinks, linkcolor=blue]{hyperref}

\usepackage{bm}
\usepackage{bbm}
\usepackage{url}

\usepackage[utf8]{inputenc}
\usepackage{mathtools,amssymb}
\usepackage{esint}
\usepackage{tikz}
\usepackage{dsfont}
\usepackage{relsize}
\usepackage{url}
\usepackage{xcolor}
\usepackage{graphicx}
\usepackage{mathrsfs}
\usepackage[shortlabels]{enumitem}
\usepackage{lineno}
\usepackage{amsmath}
\usepackage{enumitem}
\usepackage{amsthm} 
\usepackage{verbatim}
\usepackage{dsfont}
\numberwithin{equation}{section}

\allowdisplaybreaks

\mathtoolsset{showonlyrefs}

\graphicspath{{images/}}

\newtheorem{theorem}{Theorem}[section]

\newtheorem{definition}[theorem]{Definition}
\newtheorem{corollary}[theorem]{Corollary}

\newtheorem{proposition}[theorem]{Proposition}
\newtheorem{example}[theorem]{Example}

\newtheorem{remark}[theorem]{Remark}

\title[Inverse Problems for the Monge--Amp\`ere Equation]{Inverse Problems for the Monge--Amp\`ere Equation:
Linearization and Nonlinear Recovery}

\author[G. Uhlmann]{Gunther Uhlmann}
\address{Department of Mathematics, University of Washington, Seattle, WA 98195-4350, USA}
\email{gunther@math.washington.edu}

\author[P. Zimmermann]{Philipp Zimmermann}
\address{Departament de Matem\`atiques i Inform\`atica, Universitat de Barcelona, Barcelona, Spain and Institute of Mathematics, École Polytechnique Fédérale de Lausanne (EPFL), Lausanne, Switzerland}
\email{philipp.zimmermann@epfl.ch}

\newcommand{\R}{{\mathbb R}}

\newcommand{\N}{{\mathbb N}}

\newcommand{\eps}{\varepsilon}

\newcommand{\id}{\mathrm{I}_n}

\newcommand{\sym}{\mathbb{S}^n}

\newcommand{\distr}{\mathscr{D}^{\prime}}
\DeclareMathOperator{\Div}{div} 
\DeclareMathOperator{\tr}{\text{tr}} 
\DeclareMathOperator{\cof}{\text{cof}} 

\begin{document}

\begin{abstract}
	We study inverse boundary value problems for the fully nonlinear Monge--Ampère equation
	\[
	\det D^2u=a(x,u,\nabla u)
	\]
	in a bounded domain. Motivated by Calderón-type inverse problems, we introduce a natural nonlinear Cauchy data set and investigate to what extent the nonlinearity $a$ can be recovered from boundary measurements.

	Around a strictly convex background solution, and under a sign condition ensuring invertibility of the linearized Dirichlet problem, we establish local well-posedness and differentiable dependence on the boundary data. The first variation of the nonlinear measurements is the Cauchy data of a linear elliptic operator whose principal coefficient is the cofactor matrix of the background Hessian and whose lower-order coefficients are the derivatives of $a$ in $z$ and $p$.

	An Alessandrini-type identity shows that equality of nonlinear Cauchy data implies equality of the corresponding linearized measurements. For a general background this is a conditional reduction to an anisotropic Calderón-type inverse problem. Subject to a uniqueness theorem for the resulting linear coefficient class, one recovers
	\[
	\partial_z a(x,\Phi,\nabla\Phi),
	\qquad
	\nabla_p a(x,\Phi,\nabla\Phi)
	\]
	along the background jet. For a common background, higher-order linearization gives integral identities for higher derivatives of the nonlinearity; recovery from these identities is conditional on an explicitly stated product-density property.

	We finally obtain unconditional results in several special settings. For the inverse source equation $\det D^2u=f(x)$, equality of the local nonlinear Cauchy data around compatible uniformly convex backgrounds yields exact recovery of $f$ in dimension two; the same conclusion holds in dimensions $n\ge3$ under the real-analytic anisotropic hypotheses. We compare this local first-linearization result with a recent planar small-data approach based on a nondivergence-form metric reduction and with an independent all-dimensional large-data approach reducing the problem to the Euclidean X-ray transform. Furthermore, we show that for semilinear nonlinearities with background solution $\Phi(x)=|x|^2/2$, classical Schrödinger inverse results recover the available $z$-jet. For equilibrium optimal-transport nonlinearities, the linearization is an isotropic conductivity equation and the transport density is uniquely determined.

		\medskip



	\end{abstract}
	\maketitle

	\tableofcontents

\section{Introduction}
\label{sec:introduction}

Inverse boundary value problems aim to determine unknown coefficients or nonlinearities in a partial differential equation from measurements performed on the boundary of the underlying domain. Since Calderón's seminal work \cite{Calderon} on electrical impedance tomography, such problems have been extensively studied for linear and quasilinear elliptic equations \cite{Is1,IsNa,KrUh,KrUh2,CarFe1,Is2,SunUh,CarFeKiKrUh,HeSun,SaZh,carstea2024two,cârstea2025reconstructioncoefficientsdoublephase} and have led to a rich interaction between partial differential equations, differential geometry, and harmonic analysis.

In this article we study inverse problems for the nonlinear Monge--Ampère equation
\begin{equation}
\label{eq:MA}
\det D^2u=a(x,u,\nabla u)
\qquad\text{in }\Omega,
\end{equation}
where $\Omega\subset\mathbb R^n$ is a bounded domain and
\[
a:\Omega\times\mathbb R\times\mathbb R^n\to(0,\infty)
\]
is an unknown nonlinearity.

Motivated by Calderón's inverse conductivity
problem, we ask:

\medskip

\emph{Can one recover the jet of the nonlinearity $a$ along a background solution $\Phi$ from boundary measurements of \eqref{eq:MA}?}

\medskip

Our answer has multiple layers. For a general background we establish a rigorous reduction to a linear anisotropic inverse problem and derive higher-order integral identities; recovery of the complete linearized coefficient triple is conditional on the corresponding linear uniqueness and product-density statements. For the pure source equation, equality of local nonlinear Cauchy data around compatible backgrounds gives exact recovery in dimension two, and in dimensions $n\ge3$ under real-analytic anisotropic hypotheses. The planar argument uses only the first linearization: known anisotropic conductivity uniqueness up to a boundary-fixing diffeomorphism combines with the divergence-free cofactor structure to remove the gauge. For the quadratic and equilibrium backgrounds treated below, standard isotropic inverse results give further unconditional recovery statements. The arbitrary smooth cofactor-Hessian uniqueness problem in dimensions $n\ge3$ is reserved for future work.

The Monge--Ampère equation occupies a distinguished position among fully nonlinear elliptic equations. Besides its importance in geometric analysis, it plays a central role in optimal transport, prescribed Jacobian equations, geometric optics, and several problems in differential geometry. For general accounts of the Monge--Ampère equation and its geometric
background, see
\cite{nirenberg1974monge,aubin1981real,trudinger2008monge,aubin2012nonlinear,Alessio-Guido,Alessio,le2024analysis}.
Dirichlet and Neumann problems, variational methods, and related existence
questions are treated in
\cite{bakelman1983variational,lions1985two,lions1986neumann,forzani2004properties,jiang2014dirichlet,jiang2015neumannproblemmongeamperetype}.
For the linearized Monge--Ampère equation and its applications, we refer to
\cite{caffarelli1997properties,savin2010liouville,gonzálezsanz2024linearizationmongeampereequationsstatistical}.
Further geometric, optimal-transport, physical, and numerical applications
may be found in
\cite{gundersen1967monge,guan1996monge,caffarelli2010free,anedda2012problems,Brix_2015}.

At the same time, the fully nonlinear nature of the Monge--Ampère equation makes inverse problems considerably more challenging than in the semilinear or quasilinear setting. In particular, neither the equation nor the associated boundary measurements admit an obvious linear structure.

A key feature of the Monge--Ampère operator is that it possesses a divergence structure through the cofactor matrix $\cof D^2u$. The central idea of this article is that, although the equation is fully nonlinear, its inverse problem can be systematically analyzed through linearization around uniformly convex background solutions. More precisely, if $\Phi$ is a uniformly convex solution of \eqref{eq:MA}, then the linearization is governed by the operator
\[
L_\Phi v
=
\Div(\mathcal C_\Phi\nabla v)
-
b_\Phi\cdot\nabla v
-
q_\Phi v,
\]
where
\[
\mathcal C_\Phi=\cof D^2\Phi,
\qquad
b_\Phi=\nabla_pa(x,\Phi,\nabla\Phi),
\qquad
q_\Phi=\partial_za(x,\Phi,\nabla\Phi).
\]
Consequently, the inverse problem for the nonlinear Monge--Ampère equation is naturally connected to Calderón-type inverse problems for anisotropic elliptic operators.

\subsection{Main results}

The following statement summarizes the reduction principle; its complete hypotheses are given in Theorems~\ref{thm: local linearization MA} and~\ref{thm: reduction-principle-full}.

\begin{theorem}
	For $j=1,2$, let $\Phi_j$ be a uniformly convex solution of
	\[
	\det D^2\Phi=a_j(x,\Phi,\nabla\Phi)
	\quad\text{in }\Omega.
	\]
	Assume that the linearized Dirichlet problems are invertible and that the boundary traces and coefficients satisfy the compatibility conditions of Theorem~\ref{thm: reduction-principle-full}. Suppose that $a_1$ and $a_2$ give rise to the same local nonlinear Cauchy data set at the respective backgrounds.

	Then the Dirichlet-to-Neumann maps related to the linearized operator
	\[
	L^j_{\Phi_j} v
	=
	\Div(\mathcal C_{\Phi_j} \nabla v)
	-
	b^j_{\Phi_j}\cdot \nabla v
	-
	q^j_{\Phi_j} v
	\]
	coincide, where
	\[
	\mathcal C_{\Phi_j}=\cof D^2\Phi_j,
	\qquad
	b^j_{\Phi_j}=\nabla_p a_j(x,\Phi_j,\nabla\Phi_j),
	\qquad
	q^j_{\Phi_j}=\partial_z a_j(x,\Phi_j,\nabla\Phi_j).
	\]
	In particular, whenever the corresponding anisotropic inverse problem is uniquely solvable in the coefficient class under consideration, the nonlinear Cauchy data determine the first-order derivatives
	\[
	\partial_z a_j(x,\Phi_j,\nabla\Phi_j),
	\qquad
	\nabla_p a_j(x,\Phi_j,\nabla\Phi_j)
	\]
	along the background jet
	\[
	\bigl(x,\Phi_j(x),\nabla\Phi_j(x)\bigr).
	\]
	In the pure conductivity case
	\[
	b_{\Phi_j}^j=q_{\Phi_j}^j=0,
	\qquad j=1,2,
	\]
	the principal coefficient is determined exactly in either of the following
	settings: $n=2$, under the hypotheses of the planar anisotropic conductivity
	theorem \cite{astala2005anisotropic}; or $n\ge3$, when the principal
	coefficients, domain, and boundary satisfy the real-analytic hypotheses of
	\cite{lee1989determining}. In both cases the relevant linear theorem first
	determines $\mathcal C_{\Phi_j}$ up to a boundary-fixing diffeomorphism, and
	the divergence-free rigidity of Remark~\ref{rem: gauge rigidity} forces this
	diffeomorphism to be the identity. This exact conclusion concerns the
	principal coefficient and does not remove possible gauges for nonzero
	$b_{\Phi_j}^j$ and $q_{\Phi_j}^j$.
	
    Moreover, if $\Phi\vcentcolon = \Phi_1=\Phi_2$, the first-order coefficients have been identified, and the product-density implication in Theorem~\ref{thm: conditional kth derivative recovery} holds, higher-order linearization determines
	\[
	d_{(z,p)}^k a_j(x,\Phi,\nabla\Phi),
	\qquad 2\le k\le m+1,
	\]
    to the order allowed by the assumed regularity. A full Taylor-series conclusion requires these hypotheses at every order and an analyticity assumption.
\end{theorem}

The pure source equation admits stronger conclusions in dimension two and in
the real-analytic higher-dimensional setting. If
\[
a_j(x,z,p)=f_j(x),
\]
then $b_{\Phi_j}^j=q_{\Phi_j}^j=0$, so the linearized operators are pure
conductivity operators with divergence-free principal coefficients
$\mathcal C_{\Phi_j}=\cof D^2\Phi_j$. The planar anisotropic Calder\'on theorem
determines these tensors up to a boundary-fixing diffeomorphism
\cite{astala2005anisotropic}; the divergence-free rigidity (see
Remark~\ref{rem: gauge rigidity}) forces the diffeomorphism to be the identity.
Consequently, Corollary~\ref{cor: source uniqueness planar analytic} gives
\[
f_1=f_2\qquad\text{in }\Omega
\]
from equality of the local nonlinear Cauchy data, under the stated background
and boundary compatibility hypotheses. This argument is specific to the pure
conductivity reduction and does not remove lower-order gauges when
$b_{\Phi_j}^j$ or $q_{\Phi_j}^j$ is nonzero. The same proof gives exact source
recovery in dimensions $n\ge3$ when the hypotheses of the real-analytic
anisotropic conductivity theorem \cite{lee1989determining} are satisfied.
This is a local perturbative result around compatible backgrounds. Its relation
to the planar small-data approach of Liimatainen and Lin \cite{liimatainen2025inverseproblemmongeampereequation} as well as to the
all-dimensional large-data theorem of Cârstea and Ghosh \cite{cârstea2026inversesourceproblemmongeampere} is discussed in
Section~\ref{sec: source problems}.

Our analysis of the inverse nonlinear Monge--Ampère problem proceeds in three steps. First, we develop a local nonlinear theory around a uniformly convex background solution $\Phi$. We establish local well-posedness of the Dirichlet problem for boundary values close to $\Phi|_{\partial\Omega}$, prove smooth dependence of solutions on the boundary data, and derive explicit formulas for the higher-order derivatives of the solution operator.

Second, we introduce a nonlinear Cauchy data set associated with the Monge--Amp\`ere equation and study its linearization. We show that the first variation of the nonlinear measurements coincides with the Cauchy data set of the linearized operator $L_\Phi$, which in turn determines the (conormal) DN data. Using an Alessandrini-type identity, we then establish a reduction principle showing that equality of nonlinear Cauchy data implies equality of the corresponding linearized measurements. Consequently, the nonlinear inverse problem is reduced to an anisotropic Calder\'on-type inverse problem, whose anisotropy is prescribed by the cofactor matrix
\[
\mathcal C_\Phi=\cof D^2\Phi.
\]

Third, under the additional assumption $\Phi_1=\Phi_2$, we derive higher-order linearization identities. Assuming the product-density implication stated in Theorem~\ref{thm: conditional kth derivative recovery}, these identities recover the quantities
\[
d_{(z,p)}^k a(x,\Phi,\nabla\Phi),
\qquad k\ge2,
\]
along the background jet. If these hypotheses hold at every order and the nonlinearities are analytic in $(z,p)$, this yields equality of their Taylor expansions along
\[
\bigl(x,\Phi(x),\nabla\Phi(x)\bigr).
\]

The abstract theory is then applied in several concrete settings. For the two-dimensional source equation, anisotropic conductivity uniqueness and divergence-free rigidity give exact recovery. For general fixed backgrounds with lower-order terms, recovery remains subject to uniqueness modulo all relevant gauges for the complete linearized coefficient triple. In semilinear situations at a quadratic background, the linearized equation reduces to a Schr\"odinger equation and the inverse problem falls within the framework of the classical Calder\'on problem.

We further investigate nonlinearities arising in optimal transport. In the equilibrium configuration
\[
\Phi(x)=\frac{|x|^2}{2},
\]
the linearized Monge--Amp\`ere equation reduces to a conductivity equation. Consequently, the nonlinear inverse problem reduces to the classical Calder\'on conductivity problem, yielding uniqueness results for a natural class of transport nonlinearities.

\subsection{Relation to earlier work}

Higher-order linearization methods, first introduced in \cite{Is1}, have proved remarkably successful in inverse problems for semilinear and quasilinear equations. Starting from the pioneering works of Kurylev, Lassas and Uhlmann \cite{kurylev2018inverse} and many subsequent developments, these methods allow one to recover nonlinear coefficients by differentiating boundary measurements with respect to small parameters. The reader may find in \cite{hintz2022inverse} an extended list of references employing the higher-order linearization method.

In contrast, inverse problems for fully nonlinear elliptic equations remain largely unexplored. The main difficulty is that fully nonlinear equations generally lack a canonical Dirichlet-to-Neumann map and are not naturally connected to linear inverse problems. The Monge--Ampère equation is particularly well suited for overcoming these difficulties because of its divergence structure. 

During the preparation of the present manuscript, two related preprints
appeared independently, which study the inverse source problem
\[
    \det D^2u=f(x).
\]
Liimatainen and Lin
\cite{liimatainen2025inverseproblemmongeampereequation} propose a planar
local small-data proof combining first- and second-order linearization with
complex-analytic tools. Their approach first uses a nondivergence-form metric
reduction and then determines the remaining conformal factor by means of the
second linearization. Their boundary determination
provides the additional compatibility needed to normalize this factor, while
the cofactor-conductivity structure used here leads to a complementary direct
first-linearization proof. We take a closer look of their relations in
Section~\ref{sec: source problems}. Cârstea and Ghosh
\cite{cârstea2026inversesourceproblemmongeampere} obtain uniqueness in every
dimension $n\ge2$ from a family of large boundary values, reducing the problem
to the Euclidean X-ray transform. Both articles work in a higher regularity setting than this article.

The scope of the present article is different. We treat the general
nonlinearity $a(x,u,\nabla u)$ and develop a local framework around prescribed
uniformly convex backgrounds satisfying the hypotheses of the local
forward problem. The main contribution is a systematic reduction of the nonlinear inverse problem to anisotropic linear inverse
problems, together with higher-order identities that probe derivatives of the
nonlinearity along the background jet. For the source equation, our planar
result uses only the first linearization: boundary compatibility converts the
linearized normal data into conormal data for
$\cof D^2\Phi$, and the Piola identity removes the boundary-fixing
diffeomorphism gauge. A detailed comparison with the two recent source results
is given in Section~\ref{sec: source problems}. 

\subsection{Structure of the article}

The article is organized as follows.

In Section~\ref{sec: well-posedness-dn}, we establish the local well-posedness theory for the Monge--Amp\`ere equation near a uniformly convex background solution. We prove smooth dependence of solutions on the boundary data and derive higher-order linearization formulas for the associated solution operator.

In Section~\ref{sec: nonlinear Cauchy}, we introduce the nonlinear Cauchy data set associated with the Monge--Amp\`ere equation and study its first and higher-order linearizations. In particular, we show that the first variation of the nonlinear Cauchy data coincides with the Cauchy data set of the linearized operator.

Section~\ref{sec: Alessandrini identity} is devoted to the derivation of an Alessandrini-type identity for the linearized equation. This identity forms the fundamental link between the nonlinear inverse problem and the corresponding linear inverse problem.

In Section~\ref{sec: reduction linear ip for gradient}, we establish a reduction principle showing that equality of nonlinear Cauchy data implies equality of the corresponding linearized Cauchy data. As a consequence, the recovery of first-order derivatives of the nonlinearity is reduced to an anisotropic Calder\'on-type inverse problem for the linearized operator.

Section~\ref{sec: higher order integral id fixed background} develops higher-order integral identities obtained from the higher-order linearization procedure. These identities provide access to higher-order derivatives of the nonlinearity along a fixed background jet.

The inverse problem applications are collected in Section~\ref{sec: Uniqueness}. We first separate the diffeomorphism gauge of the principal conductivity from possible lower-order gauges and explain the exact planar conclusion for pure cofactor-Hessian conductivities. We then derive conditional higher-order consequences and unconditional recovery results for special semilinear, quadratic, and equilibrium backgrounds.

In Section~\ref{sec: source problems}, we prove exact uniqueness for the
inverse source equation in dimension two and in the real-analytic
higher-dimensional setting. We then compare the hypotheses and measurement
regimes with recent source results and formulate the remaining smooth
cofactor-Hessian problem in dimensions $n\ge3$.

Finally, Appendix~\ref{sec: appendix examples} contains examples of nonlinearities covered by our framework and discusses several structural classes arising in applications, including examples motivated by optimal transport theory.

\medskip

\noindent
\textbf{Notation.} Throughout this article we use the following standard notation. If $\Omega\subset \R^n$ is an open set and $f\colon\Omega\to \R^m$ a smooth function, then we write $\nabla f$ for its Jacobian, which has components $(\nabla f)_{ij}=\partial_j f_i$. Moreover, we define iteratively the higher power gradients $\nabla^k f$, $k\in\N$, by the formula $\nabla^k f=\nabla^{k-1}(\nabla f)$. Note that the Hessian matrix $D^2 u$ of a function $u\colon \Omega\to\R$ is the transpose of $\nabla^2 u$, that is, $D^2 u= (\nabla^2 u)^T$. For any bounded domain $\Omega\subset \R^n$, the space of $k$-times (continuously) differentiable functions on $\Omega$ is denoted by $C^k(\Omega)$. Furthermore, $C^k(\overline{\Omega})$ stands for the set of all $u\in C^k(\Omega)$ such that $\nabla^{\ell} u$, $0\leq \ell\leq k$, can be continuously extended to $\overline{\Omega}$. This space becomes a Banach space when we endow it with the norm
\begin{equation}
\label{eq: Ck norm}
    \|u\|_{C^k(\overline{\Omega})}\vcentcolon =\sum_{0\leq \ell\leq k}\|\nabla^{\ell} u\|_{L^{\infty}(\Omega)}.
\end{equation}
For any $0<\alpha\leq 1$ and $k\in\N_0$, we denote by $C^{k,\alpha}(\overline{\Omega})$ the space of $C^k(\overline{\Omega})$ functions such that $\nabla^k u$ is $\alpha$-H\"older continuous, that is, its H\"older seminorm is finite,
\begin{equation}
\label{eq: holder seminorm}
    [\nabla^k u]_{C^{k,\alpha}(\overline{\Omega})}\vcentcolon = \sup_{x\neq y\in\Omega}\frac{|\nabla^k u(x)-\nabla^k u(y)|}{|x-y|^{\alpha}}<\infty. 
\end{equation}
By standard results it follows that $C^{k,\alpha}(\overline{\Omega})$ with norm
\[
    \|u\|_{C^{k,\alpha}(\overline{\Omega})}\vcentcolon = \sum_{0\leq \ell\leq k}\|\nabla^{\ell} u\|_{L^{\infty}(\Omega)}+[\nabla^k u]_{C^{k,\alpha}(\overline{\Omega})}
\]
is a Banach space and 
\[
    C^{k,\alpha}_0(\overline{\Omega})\vcentcolon = \{u\in C^{k,\alpha}(\overline{\Omega})\,;\,u|_{\partial\Omega}=0\}
\]
is a closed subspace, which makes itself into a Banach space. If $\Omega\subset\R^n$ is not bounded, then we may define the above spaces exactly in the same way and they keep the same properties, when we require additionally that the $C^k$-norms are finite.

We also need to speak about H\"older continuous functions on the boundary $\partial\Omega$ of a bounded domain $\Omega\subset\R^n$. We write $\partial \Omega \in C^{k,\alpha}$ for $k\in\N$ and $0\leq \alpha\leq 1$, and say $\Omega$ is a $C^{k,\alpha}$ domain, when $\partial\Omega$ is locally the graph of $C^{k,\alpha}$ function $\Gamma\colon \R^{n-1}\to \R$ depending only on $n-1$ coordinates of $x_1,\ldots,x_n$. Equivalently, $\partial\Omega\in C^{k,\alpha}$, if for every $x_0\in\partial\Omega$ there exists a ball $B$ centered at $x_0$ and a (global) $C^{k,\alpha}$ diffeomorphism $\psi$ from $B$ onto $V\subset  \R^n$ such that $\psi(\Omega\cap B)\subset \R^n_+$ and $\psi(\Omega\cap \partial\Omega)\subset \partial \R^n_+$. 

For a given $C^{k,\alpha}$ domain $\Omega\subset\R^n$, we say that $\varphi\in C^{k,\alpha}(\partial\Omega)$, $k\in\N$, $0\leq \alpha\leq 1$, when $\varphi$ is in every chart of class $C^{k,\alpha}$, that is, $\varphi\circ\psi^{-1}(V\cap \partial \R^n_+)$ for each $x_0\in \partial\Omega$ and every chart $(B,\psi)$. In particular, one may observe that by classical extension theorems in the half space and a partition of unity that every $\varphi\in C^{k,\alpha}(\partial\Omega)$ can be extended to a function in $C^{k,\alpha}(\overline{\Omega})$ and every function in $C^{k,\alpha}(\overline{\Omega})$ has a trace in $C^{k,\alpha}(\partial\Omega)$. Hence, $C^{k,\alpha}(\partial\Omega)$ becomes a Banach space, when endowed with the quotient norm
\[
    \|\varphi\|_{C^{k,\alpha}(\partial\Omega)}\vcentcolon = \inf_{\Phi\in C^{k,\alpha}(\overline{\Omega})\,;\,\Phi|_{\partial\Omega}=\varphi}\|\Phi\|_{C^{k,\alpha}(\overline{\Omega})}.
\]
Sometimes we will abbreviate the symbol $C^{k,\alpha}(\overline{\Omega};\R^m)$ by $C^{k,\alpha}$, but in these cases the domain and target spaces will always be clear from the context. 

We also use results from differential calculus on Banch spaces \cite{ambrosetti1995primer}. Let $X$, $Y$ be Banach spaces and suppse $U\subset X$ is an open set. We call a function $f\colon U \to Y$ \emph{differentiable} at $x\in U$ whenever it is differentiable in the Fr\'echet sense, meaning that there exists $A\in L(X,Y)$ such that
\begin{equation}
\label{eq: derivative}
	f(x+h)=f(x)+Ah+o(\|h\|_X)
\end{equation}
as $\|h\|_X \to 0$. If it exists, then we set $df(x)\vcentcolon = A$ and call it (Fr\'echet) differential at $x$. As in calculus in $\R^n$, if $f$ is differentiable at $x$, then all its directional derivatives or the Gateaux differential exist. In fact, for any $h\in X$ we have
\begin{equation}
\label{eq: Gateaux der}
	d_G f(x)h\vcentcolon = \left.\frac{d}{dt}\right|_{t=0}f(x+th)=df(x)h.
\end{equation}
Conversely, if $d_Gf(x)\in L(X,Y)$ exist in a neighborhood of $x_0$ and $x\mapsto d_G f$ is continuous at $x_0$, then $f$ is differentiable at $x_0$.

Moreover, for any $k\in\N$, we may define iteratively higher-order Fr\'echet derivatives
\begin{equation}
\label{eq: higher order derivatives}
	d^kf(x)=d(d^{k-1}f)\in L(X,L_{k-1}(X,Y))\simeq L_k(X,Y),
\end{equation}
where $L_k(X,Y)$ denotes the space of $k$-linear functions from the $k$-fold product $X\times \ldots \times X$ to $Y$. The value of $d^kf(x)$ at $(h_1,\ldots,h_k)\in X\times \ldots\times X$ will be denoted by
\[
	d^kf(x)[h_1,\ldots,h_k]\quad\text{or}\quad d^kf(x)[h_j]_{j=1,\ldots,k}
\]
for all $(h_j)_{j=1,\ldots,k}\subset X$. Indeed, it turns out that the derivative $d^k f(x)$ is symmetric, when $f$ is $k$-times differentiable. Furthermore, we say that $f\in C^k(U;Y)$, when $f$ is $k$ times differentiable in $U$ and the $k$-th derivative $d^k F$ is continuous from $U$ to $L_k(X,Y)$.

Similarly as in the case $X=\R^n$ and $Y=\R^m$, we use partial derivatives for functions from a product space $X\times Y$ to another Banach space $Z$. More concretely, we say that $f\colon X\times Y\to Z$ is differentiable with respect to $x$ at $(x_0,y_0)$, when the linear bounded map
\[
	d_x f(x_0,y_0)\vcentcolon = d(f(\cdot,y_0))(x_0)\in L(X,Z)
\]
exists. An analog convention applies to partial derivatives in the $y$ direction. Again by iteration, we may define higher order partial derivatives.

\medskip

Next, let us recall some standard facts from linear algebra. If $A\in \R^{n\times n}$ is a $n\times n$ matrix, then $\cof A\in \R^{n\times n}$ denotes its cofactor matrix, whose components are given by
\begin{equation}
\label{eq: cofactor matrix}
    (\cof A)_{ij}=(-1)^{i+j}\det \widetilde{A}_{ij}
\end{equation}
and $\widetilde{A}_{ij}$ is the matrix $A$ without the $i$-th row and $j$-th column. Using the cofactor matrix, the Laplace expansion theorem  and Jacobi's formula read as followos:
\begin{equation}
\label{eq: Laplace}
    \det A=\sum_{1\leq i\leq n}A_{ij}(\cof A)_{ij}=(A^T(\cof A))_{jj},
\end{equation}
and
\begin{equation}
\label{eq: Jacobi}
    \partial_t\det A=\tr ((\cof A)^T \partial_t A)
\end{equation}
Formula \eqref{eq: Laplace} holds for any matrix $A=(A_{ij})_{ij}\in\R^{n\times n}$ and \eqref{eq: Jacobi} for differentiable matrix fields $t\mapsto A(t)\in \R^{n\times n}$. The cofactor matrix has several nice properties, including the scaling $\cof (\lambda A)=\lambda^{n-1}\cof A$ and Cramer's rule
\begin{equation}
\label{eq: Cramer}
    A^{-1}=(\det A)^{-1}(\cof A)^T,
\end{equation}
which holds for invertible matrices. 

We denote by $\sym$ the space of symmetric matrices. Furthermore, we recall that $A\in \sym$ is nonnegative, when 
\[
	\det A\geq 0\quad\text{or equivalently}\quad \xi\cdot A\xi\geq 0\text{ for all }\xi\neq 0.
\]
For two matrices $A,B\in\sym$ we write
\[
    A\geq B\,\Leftrightarrow \, A-B\text{ is nonnegative}.
\] 
Since every symmetric matrix is orthogonally diagonalizable, we have the following estimate for any symmetric, nonnegative definite matrix $A$,
\[
	\xi \cdot A\xi=\xi\cdot ODO^T\xi=O\xi\cdot D(O\xi)=\sum_{j=1}^n \lambda_j (O\xi)_j^2\geq\lambda_1 |O\xi |^2=\lambda_1|\xi|^2.
\]
Here, $O\in \R^{n\times n}$ is an orthogonal matrix and $D=\text{diag}(\lambda_1,\ldots,\lambda_n)$, with $(\lambda_j)_{j=1,\ldots,n}$ denoting its increasingly ordered eigenvalues, such that $A=ODO^T$. Similarly, we have
\[
		\xi \cdot A\xi=\sum_{j=1}^n \lambda_j (O\xi)_j^2\leq \lambda_n|\xi|^2.
\]
Hence
\begin{equation}
\label{eq: bounds sym nonneg def}
	\lambda_1\id\leq A \leq \lambda_n \id \quad\text{for any}\quad A\in\sym.
\end{equation}
Furthermore, we say that the matrix field $A=A(x)\in\sym$ defined for $x\in\overline{\Omega}$ is uniformly elliptic, when 
\begin{equation}
\label{eq: uniform ellipticity}
        C^{-1}\id\leq A(x)\leq C \id
\end{equation}
for some constant $C>0$, or equivalently when there exists a $C>0$ such that 
\[
	C^{-1}\leq \lambda^x_1\leq \ldots\leq \lambda_n^x\leq C,
\]
where $(\lambda_j^x)_{j=1,\ldots,n}$ are the increasingly ordered eigenvalues of $A(x)$.

Throughout this paper, we use the following conventions. Let $U\subset\R^n$ be an open set. A function $v\in C^2(U)$ is called \emph{uniformly convex} if there exists a constant $c>0$ such that
\[
    D^2v(x)\geq cI
    \qquad\text{for all }x\in U.
\]
If, in addition, $v\in C^2(\overline U)$, then the same inequality holds for all $x\in\overline U$ by continuity.

Moreover, a bounded domain $\Omega\subset\R^n$
of class $C^2$ is called \emph{uniformly convex}
if there exists a constant $\kappa>0$ such that
\[
    \nu(x)\cdot(y-x)\leq-\kappa|y-x|^2
    \qquad
    \text{for all }x\in\partial\Omega,\quad
    y\in\overline\Omega,
\]
where $\nu$ denotes the outward unit normal
to $\partial\Omega$.

\section{Local well-posedness of the nonlinear Monge--Amp\'ere equation}
\label{sec: well-posedness-dn}

We start with the following convenient definition.
\begin{definition}[Linearized operator]
\label{def: linearized operator}
    Let $\Omega\subset\R^n$ be a bounded domain, let $a$ be differentiable in $(z,p)$, and let $\Phi\in C^2(\overline{\Omega})$. We define the \emph{$\Phi$-linearized Monge--Amp\'ere operator} $L_\Phi$ by
    \begin{equation}
    \label{eq: linearized operator}
  L_\Phi v \vcentcolon = \tr(\mathcal{C}_\Phi D^2 v)
- b_\Phi \cdot \nabla v - q_\Phi v,\quad v\in C^2(\Omega),
    \end{equation}
    where 
    \begin{equation}
    \label{eq: coefficients}
        \mathcal{C}_\Phi \vcentcolon = \cof D^2\Phi,\quad b_\Phi \vcentcolon = \nabla_p a(x,\Phi,\nabla \Phi),\quad\text{and}\quad q_\Phi \vcentcolon = \partial_z a(x,\Phi,\nabla \Phi).
    \end{equation}
\end{definition}

\begin{remark}
	\label{rem: divergence form}
	The linearized operator can equivalently be written in divergence form. Indeed,
	for smooth $\Phi$, the Piola identity \cite[Section~8.1.4]{EvansPDE} gives
	\[
	\Div(\cof D^2\Phi)=0.
	\]
	Although this identity is classical for smooth functions, it remains valid in the
	distributional sense under the assumption $\Phi\in C^{2,\alpha}(\overline\Omega)$.
	Indeed, by local mollification one obtains
	\[
	\Div(\cof D^2\Phi)=0
	\quad\text{in }\distr(\Omega).
	\]
	Next, we recall that 
	\begin{equation}
	\label{eq: divergence of matrix}
		(\Div A)_j =\sum_i\partial_i A_{ij}
	\end{equation}
	implies
	\[
		\Div(A\nabla u)=\tr(A D^2 u)+(\Div A)\cdot \nabla u
	\]
	for all sufficiently regular matrix fields $A$ and functions $u$. Furthermore, since $\Phi\in C^{2,\alpha}(\overline{\Omega})$, we have
	\[
	\cof D^2\Phi\in C^{0,\alpha}(\overline{\Omega};\R^{n\times n}).
	\]
	Hence, for every $v\in H^2_{\mathrm{loc}}(\Omega)$,
	\[
	\Div((\cof D^2\Phi)\nabla v)
	\]
	is well-defined as a distribution. Moreover, using the identity
	\[
	\Div(\cof D^2\Phi)=0
	\quad\text{in }\distr(\Omega),
	\]
	one obtains
	\[
	\Div((\cof D^2\Phi)\nabla v)
	=
	\tr((\cof D^2\Phi)D^2v)
	\quad\text{in }\distr(\Omega).
	\] 
	Hence
	\[
	L_\Phi v
	=
	\Div(\mathcal C_\Phi\nabla v)
	-
	b_\Phi\cdot\nabla v
	-
	q_\Phi v\quad \text{in }\distr(\Omega).
	\]
\end{remark}

Next, we establish the following local well-posedness result.

\begin{theorem}[Local well-posedness and linearization]
\label{thm: local linearization MA}
Let $\Omega \subset \R^n$ be a bounded domain with $C^{m+2,\alpha}$ boundary, $m \in \mathbb{N}_0$, $0<\alpha<1$, and $\phi\in C^{m+2,\alpha}(\partial\Omega)$. Let
\[
a \in C^{2m+1,\alpha}(\overline{\Omega} \times \R \times \R^n)
\]
be positive, and let $\Phi \in C^{m+2,\alpha}(\overline{\Omega})$ be a uniformly convex solution of the Dirichlet problem
\begin{equation}
\label{eq: Phi equation}
\begin{cases}
    \det D^2 \Phi = a(x,\Phi,\nabla \Phi)  &\text{in } \Omega,\\
    \Phi = \phi  &\text{on } \partial\Omega,
\end{cases}
\end{equation}
such that
\begin{equation}
\label{eq: uniform ellipticity cond}
     D^2 \Phi(x) \geq c_0\id\quad \text{in } \overline{\Omega}
\end{equation}
for some $c_0 > 0$. Furthermore, assume that 
\begin{equation}
\label{eq: nonnegativity of lowest order term}
	q_\Phi = \partial_z a(x,\Phi,\nabla \Phi) \geq 0.
\end{equation}
Then the following assertions hold true:
\begin{enumerate}[(i)]
    \item\label{prop 1: solvability monge ampere} There exists $\delta > 0$ such that for all boundary data 
$\eta \in C^{m+2,\alpha}(\partial\Omega)$ with
\[
\|\eta - \phi\|_{C^{m+2,\alpha}(\partial\Omega)} < \delta,
\]
the problem
\begin{equation}
\label{eq: MA well-posedness}
\begin{cases}
\det D^2 u = a(x,u,\nabla u) & \text{in } \Omega, \\
u = \eta & \text{on } \partial\Omega
\end{cases}
\end{equation}
admits a unique solution $u_\eta \in C^{m+2,\alpha}(\overline{\Omega})$ in a fixed neighborhood of $\Phi$. After decreasing $\delta$ if necessary, these solutions remain uniformly convex.
\item\label{prop 2: solution operator} For every integer $k$ satisfying
\[
0\le k\le m+1,
\]
    the local solution operator
    \[
     C^{m+2,\alpha}(\partial\Omega)\ni \eta \overset{\mathcal{S}}{\mapsto} u_\eta \in C^{m+2,\alpha}(\overline\Omega)
    \]
    is $C^k$ in a neighborhood of $\phi$. 
    \item \label{prop 3: linearization}
    For every $h\in C^{m+2,\alpha}(\partial\Omega)$, we have
\begin{equation}
\label{eq: derivative sol map}
d\mathcal S(\phi)h=v_h,
\end{equation}
where $v_h \in C^{m+2,\alpha}(\overline{\Omega})$ is the unique solution of
\begin{equation}
\label{eq: linearized problem}
\begin{cases}
L_\Phi v_h =0 & \text{in }\Omega,\\
v_h = h & \text{on }\partial\Omega.
\end{cases}
\end{equation}
	 \item \label{prop 4: higher order linearization} Let $(h_j)_{j=1,\ldots,k}\subset C^{m+2,\alpha}(\partial\Omega)$, $1\leq k\leq m+1$, and suppose that $\varnothing\ne J\subset \{1,\ldots,k\}$. 
	 Then
	 \[
	 U_J
	 :=
	 d^{|J|}\mathcal S(\phi)[h_j]_{j\in J}
	 \in C^{m+2,\alpha}(\overline\Omega)
	 \]
	satisfies
	\begin{equation}
	\label{eq: PDE for UJ}
	\begin{split}
		L_\Phi U_J
		&=
		-\sum_{\substack{\pi\in\mathcal P(J):\\ |\pi|\ge2}}
		d^{|\pi|}(\det)(D^2\Phi)[D^2 U_M]_{M\in \pi }\\
		&\quad
		+
		\sum_{\substack{\pi\in\mathcal P(J):\\
				 |\pi|\ge2}}
		d^{|\pi|}_{(z,p)}a(x,\Phi,\nabla\Phi)
		\big[
		(U_M,\nabla U_M)\big]_{M\in\pi}.
	\end{split}
	\end{equation}
	and 
	\[
		U_J|_{\partial\Omega}=\begin{cases}
		h_j,&\quad \text{if }J=\{j\},\\
		0,&\quad\text{if }|J|\geq 2.
		\end{cases}
	\]
\end{enumerate}
\end{theorem}

\begin{remark}[Existence of admissible background solutions]
	\label{rem: existence background}
	Let $\Omega\subset\mathbb R^n$ be a uniformly convex bounded domain of class
	$C^{4,\alpha}$, $0<\alpha<1$, and let
	$\phi\in C^{4,\alpha}(\partial\Omega)$.
	Assume that
	\[
	a\in C^{3,\alpha}(\overline{\Omega}\times\mathbb R\times\mathbb R^n)
	\]
	satisfies
	\[
	a(x,z,p)\ge \lambda>0,
	\qquad
	\partial_z a(x,z,p)\ge0
	\]
	in
	$\overline{\Omega}\times\mathbb R\times\mathbb R^n$.
	Then, by \cite[Theorem~4.13]{le2024analysis},
	there exists a unique uniformly convex solution
	\[
	\Phi\in C^{4,\alpha}(\overline{\Omega})
	\]
	of
	\[
	\begin{cases}
	\det D^2\Phi=a(x,\Phi,\nabla\Phi)
	&\text{in }\Omega,\\
	\Phi=\phi
	&\text{on }\partial\Omega.
	\end{cases}
	\]
    The proof of \cite[Theorem~4.13]{le2024analysis}, like the invertibility argument below, uses the method of continuity. In the cited theorem, this method establishes existence and uniqueness of the background solution, while below the sign condition on $\partial_z a$ is used to prove invertibility of the linearized Dirichlet operator.
\end{remark}

\begin{proof}[Proof of Theorem~\ref{thm: local linearization MA}] We show the result in three steps.\\

    \noindent\textbf{Step 1: Linearized problem.} We first address the unique solvability of the linearized problem \eqref{eq: linearized problem}. 

      \medskip

    \textbf{Step 1.a: Uniform ellipticity.} We first prove that the leading order coefficient of $L_{\Phi}$ is uniformly elliptic. 

    Using $\Phi\in C^{m+2,\alpha}(\overline{\Omega})$ and \eqref{eq: uniform ellipticity cond}, we deduce that 
    \[
    D^2 \Phi\in C^{m,\alpha}(\overline{\Omega};\R^{n\times n})\quad\text{is symmetric and invertible.}
    \]
    Thus, $D^2\Phi(x)$ is orthogonally diagonalizable for any $x\in \overline{\Omega}$. Therefore, for any $x\in \overline{\Omega}$, there exist an orthogonal matrix $O_x$ and a diagonal matrix  $D_x\vcentcolon = \text{diag}(\lambda_x^1,\ldots,\lambda_x^n)$, where 
    \[
    0<\lambda_x^1\leq \ldots\leq \lambda_x^n
    \]
    are the increasingly ordered eigenvalues of $D^2\Phi(x)$, such that
    \[
    D^2 \Phi(x)=O_x D_x O_x^T.
    \]
    Thus, we deduce from \eqref{eq: Cramer} that
    \[
    \begin{split}
    \cof D^2\Phi(x)&=(\det D^2\Phi(x))(D^2\Phi(x))^{-1}\\
    &=(\det D^2\Phi(x))O_x D_x^{-1}O_x^T.
    \end{split}
    \]
    Therefore, $\cof D^2 \Phi(x)$ is again a symmetric, positive definite matrix, whose (increasingly ordered) eigenvalues $\mu_x^j$, $j=1,\ldots,n$, are given by
    \[
  	  \mu_x^j=\frac{\det D^2\Phi(x)}{\lambda^{n-j+1}_x}=\frac{\prod_{k=1}^n \lambda_x^k}{\lambda^{n-j+1}_x}
    \]
    for $j=1,\ldots,n$. Since condition \eqref{eq: uniform ellipticity cond} implies
    \[
    	\lambda^j_x \geq c_0\quad \text{for any }j=1,\ldots,n,
    \]
    we deduce from \eqref{eq: bounds sym nonneg def} that
    \[
    \cof D^2\Phi(x)\geq \left(\prod_{j=1}^{n-1}\lambda_x^j\right)\id\geq c_0^{n-1}\id.
    \]
    Moreover, let us note that
    \[
    \lambda_x^n =\xi_n\cdot D^2\Phi(x)\xi_n\leq  |D^2\Phi(x)|,
    \]
    where $\xi_n\in \partial B_1(0)$ is an eigenvector of $D^2\Phi(x)$ corresponding to the eigenvalue $\lambda_x^n$. Hence, \eqref{eq: bounds sym nonneg def} yields
    \[
    \cof D^2\Phi(x)\leq \left(\prod_{j=2}^{n}\lambda_x^j\right)\id \leq \|D^2 \Phi\|_{L^{\infty}(\Omega)}^{n-1}\id.
    \]
    Therefore, we have shown that $\mathcal{C}_\Phi$ is uniformly elliptic.

    \medskip

    \textbf{Step 1.b: Regularity.} Next, we record the regularity of the coefficients $L_\Phi$. Observe that the cofactor map is polynomial
    in the entries of the matrix, because the same holds for the determinant and the cofactor matrix can be calculated by \eqref{eq: cofactor matrix}. Since $D^2\Phi\in C^{m,\alpha}(\overline\Omega)$ and H\"older spaces are closed under pointwise multiplication, we deduce that
    \[
    \mathcal C_\Phi=\cof D^2\Phi
    \in C^{m,\alpha}(\overline\Omega;\mathbb R^{n\times n}).
    \]
    Moreover, since
    \[
    a\in C^{m+1,\alpha}(\overline\Omega\times\mathbb R\times\mathbb R^n),
    \]
    we have
    \[
    \partial_z a,\ \nabla_p a
    \in C^{m,\alpha}(\overline\Omega\times\mathbb R\times\mathbb R^n).
    \]
    Composing these functions with the map
    \[
    x\mapsto (x,\Phi(x),\nabla\Phi(x)),
    \]
    which belongs to the class $C^{m+1,\alpha}(\overline \Omega; \R^n\times \R\times \R^n)$ and by boundedness of $\Omega$ to $C^{m,1} (\overline \Omega; \R^n\times \R \times \R^n)$, we deduce that
    \[
    q_\Phi:=\partial_z a(x,\Phi,\nabla\Phi)\in C^{m,\alpha}(\overline\Omega),
    \]
    and
    \[
    b_\Phi:=\nabla_p a(x,\Phi,\nabla\Phi)
    \in C^{m,\alpha}(\overline\Omega;\mathbb R^n).
    \]
    Indeed, recall that if $g\in C^{0,\gamma}(\R^n;\R^m)$ and $f\in C^{0,\beta}(\overline{\Omega})$ for exponents $0<\beta,\gamma\leq 1$, then $g\circ f\in C^{0, \beta \gamma}(\overline \Omega ; \R^m)$ satisfies
    \[
    	[g\circ f]_{C^{0,\beta \gamma}}\leq [g]_{C^{0,\gamma}}[f]^{\gamma}_{C^{0,\beta}}.
    \]
    Hence
    \begin{equation}
    \label{eq: regularity of coefficients}
    (\mathcal C_\Phi,b_\Phi,q_\Phi)
    \in
    C^{m,\alpha}(\overline\Omega;\mathbb R^{n\times n}\times \mathbb R^n\times \R).
    \end{equation}

    \medskip

    \textbf{Step 1.c: Well-posedness.} Finally, we establish the unique solvability of \eqref{eq: linearized problem} using the continuity method.  

    For $t\in[0,1]$, set
    \[
    L_t v
    :=
    (1-t)\Delta v
    +
    tL_\Phi v.
    \]
    Equivalently,
    \[
    L_t v
    =
    \tr(\mathcal C_tD^2 v)
    -
    t b_\Phi\cdot\nabla v
    -
    t q_\Phi v,
    \]
    where
    \[
    \mathcal C_t=(1-t)I+t\mathcal C_\Phi.
    \]
    Since $\mathcal C_\Phi$ is uniformly elliptic, the family
    $\{\mathcal C_t\}_{t\in[0,1]}$ is uniformly elliptic with constants independent
    of $t$. Moreover,
    \[
    tq_\Phi\ge 0.
    \]
    Indeed, we note that the regularity properties \eqref{eq: regularity of coefficients} ensure that
    \[
    	L_t\colon C^{m+2,\alpha}_0(\overline{\Omega})\to C^{m,\alpha}(\overline{\Omega})
    \]
    are well-defined and bounded.

    We first prove a uniform a priori estimate. By the maximum principle \cite[Theorem~3.7]{GT} applied to
    \[
    	L_t v=f\quad \text{in }\Omega\quad\text{with}\quad v\in C(\overline{\Omega})\cap C^2(\Omega),\quad v|_{\partial\Omega}=0,\quad\text{and}\quad f\in C(\overline{\Omega}),
    \]
  we obtain
    \[
    \|v\|_{L^\infty(\Omega)}
    \le C\|f\|_{L^\infty(\Omega)},
    \]
    where $C$ is independent of $t$. The global Schauder estimates for uniformly elliptic, non-divergence form
    operators with $C^{m,\alpha}$--coefficients then give
    \begin{equation}
    \label{eq: Schauder general m}
    \|v\|_{C^{m+2,\alpha}(\overline\Omega)}
    \le
    C\|L_t v\|_{C^{m,\alpha}(\overline\Omega)},
    \end{equation}
    with $C>0$ independent of $t$. In fact, in the case $m=0$, we may apply \cite[Theorem~6.6]{GT} together with the preceding maximum-principle estimate to find
    \[
    	  \|v\|_{C^{2,\alpha}(\overline\Omega)}\leq  C\bigl(
    	\|L_t v\|_{C^{
    			0,\alpha}(\overline\Omega)}
    	+
    	\|v\|_{L^\infty(\Omega)}
    	\bigr)
    	\le C\|L_t v\|_{C^{0,\alpha}(\overline\Omega)}.
    \]
    Similarly, in the case $m\geq 1$, we apply \cite[Theorem~6.19]{GT} and the maximum principle to find \eqref{eq: Schauder general m}.

    At $t=0$, the operator
    \[
    L_0=\Delta:C^{m+2,\alpha}_0(\overline\Omega)\to C^{m,\alpha}(\overline\Omega)
    \]
    is surjective. This follows from \cite[Theorem~2.30]{Xavier-book} and the direct method of the calculus of variations.

    Now, the uniform a priori estimate \eqref{eq: Schauder general m}, the surjectivity of $L_0$, and the method of continuity \cite[Theorem~5.2]{GT}
    therefore imply that $L_t$ is surjective for every $t\in[0,1]$. In
    particular, by \eqref{eq: Schauder general m},
    \begin{equation}
    \label{eq: L Phi isomorphism}
    L_\Phi=L_1: C^{m+2,\alpha}_0(\overline\Omega)\to C^{m,\alpha}(\overline\Omega)
    \end{equation}
    is an isomorphism. Now, since $\partial\Omega\in C^{m+2,\alpha}$, we can extend every boundary condition $\psi\in C^{m+2,\alpha}(\partial\Omega)$ to a function $\Psi\in C^{m+2,\alpha}(\overline{\Omega})$. Thus, for given $f\in C^{m,\alpha}(\overline{\Omega})$ and $\psi \in C^{m+2,\alpha}(\partial\Omega)$, $u$ solves
    \begin{equation}
    \label{eq: full problem}
    \begin{cases}
    L_\Phi u =f& \text{in }\Omega,\\
    u = \psi & \text{on }\partial\Omega.
    \end{cases}
    \end{equation}
    if and only if $v=u-\Psi$ solves
    \begin{equation}
    \label{eq: linear source problem}
    \begin{cases}
    L_\Phi v =f-L_\Phi \Psi & \text{in }\Omega,\\
    v = 0 & \text{on }\partial\Omega.
    \end{cases}
    \end{equation}
    in $C^{m+2,\alpha}(\overline{\Omega})$. The latter problem is well-posed, by \eqref{eq: L Phi isomorphism} and the fact that $L_\Phi \Psi \in C^{m,\alpha}(\overline{\Omega})$. Hence, we deduce that 
   \begin{equation}
   \label{eq: isomorphism full problem}
   C^{m+2,\alpha}(\overline{\Omega})	\ni u\mapsto (L_\Phi u, u|_{\partial\Omega})\in C^{m,\alpha}(\overline{\Omega})\times C^{m+2,\alpha}(\partial\Omega)
   \end{equation}
   is an isomorphism and, in particular, the Dirichlet problem \eqref{eq: linearized problem} is well-posed.

    \medskip

    \noindent\textbf{Step 2: Nonlinear problem.}

    We define
    \[
    	F\colon C^{m+2,\alpha}(\partial\Omega)\times C^{m+2,\alpha}(\overline{\Omega})\to C^{m,\alpha}(\overline{\Omega})\times C^{m+2,\alpha}(\partial\Omega)
    \] 
    by
    \[
    	F(\eta,u)\vcentcolon = (\det D^2 u -a(\cdot,u,\nabla u),u|_{\partial\Omega}-\eta).
    \]
    For later convenience, we set
    \begin{equation}
    \label{eq: defs of G and H}
    	 G(u) \vcentcolon = \det D^2 u \quad\text{and}\quad
    	 H(u) \vcentcolon = a(\cdot,u,\nabla u)
    \end{equation}
    for $u\in C^{m+2,\alpha}(\overline{\Omega})$. By assumption, we have
    \begin{equation}
    \label{eq: F vanishes}
    	F(\phi,\Phi)=0.
    \end{equation}
    By Jacobi's formula \eqref{eq: Jacobi}, we have
    \begin{equation}
    \label{eq: first derivative of G}
    	d G(u)h=\left.\frac{d}{dt}\right|_{t=0}\det (D^2 u+tD^2 h)=\tr ((\cof D^2 u)^T D^2 h)
    \end{equation}
    for any $u,h\in C^{m+2,\alpha}(\overline{\Omega})$. Recalling that for all invertible, differentiable matrix functions $t\mapsto A(t)\in\R^{n\times n}$, we have 
    \begin{equation}
    \label{eq: identities for invertible}
    	(\cof A)^T=(\det A)A^{-1}\quad \text{and}\quad \partial_t A^{-1}=-A^{-1}(\partial_t A) A^{-1}, 
    \end{equation}
    we get
    \[
    \begin{split}
    	d^2 G(u)(h,k)&=d(dG(u)h)k\\
    	&=\left.\frac{d}{dt}\right|_{t=0}\tr((\cof (D^2 u+t D^2 k))^T D^2 h)\\
    	&=\tr((\cof D^2 u)^T D^2 k)\tr ((D^2 u)^{-1}D^2 h)\\
    	&\quad -(\det D^2 u)\tr((D^2 u)^{-1}(D^2 k)(D^2 u)^{-1}(D^2 h))
    \end{split}
    \]
    for any $h,k\in C^{m+2,\alpha}(\overline\Omega)$ and any $u\in C^{m+2,\alpha}(\overline\Omega)$ sufficiently close to $\Phi$ such that $\det D^2 u >0$. Iteratively, we find that $G$ is $C^{\infty}$ in the Fr\'echet sense. 

    Furthermore, since we assume that $a\in C^{2m+1,\alpha}$, the standard composition theorem
    for Hölder spaces implies that $H(u)=a(\cdot,u,\nabla u)$ is of class $C^{m+1}$ as a map
    \[
    H\colon C^{m+2,\alpha}(\overline\Omega)\to C^{m,\alpha}(\overline\Omega).
    \]
    Indeed, this is a consequence of the observation that the $(m+1)$-st Fr\'echet derivative contains terms schematically of the form
    \[
    	\partial_{(z,p)}^{m+1}a(x,u,\nabla u)\prod_{j=1}^{m+1}(h_j,\nabla h_j),
    \]
    which continuously maps from $u\in C^{m+2,\alpha}(\overline{\Omega})\hookrightarrow C^{m+1,1}(\overline{\Omega})$ to $C^{m,\alpha}(\overline{\Omega})$, under the condition $a\in C^{2m+1,\alpha}$. Moreover, its first derivative is given by
    \begin{equation}
    \label{eq: first derivaitve of H}
    	d H(u)h=\left.\frac{d}{dt}\right|_{t=0} H(u+t h)=\nabla_p a(\cdot, u,\nabla u)\cdot \nabla h+a_z (\cdot, u,\nabla u)h
    \end{equation}
    for all $u,h\in C^{m+2,\alpha}(\overline{\Omega})$. 

    Since derivatives in the $\eta$ coordinate of $F$ are obviously $C^{\infty}$, we see that $F$ is of class $C^{m+1}$, and
    \[
    \begin{split}
    	d_{u} F(\phi,\Phi)h&=\left.\frac{d}{dt}\right|_{t=0} F(\phi,\Phi+th)=(L_\Phi h,h|_{\partial\Omega})
    \end{split}
    \]
    for all $h\in C^{m+2,\alpha}(\overline{\Omega})$, 
    where $L_\Phi$ is the $\Phi$ linearized operator. The last equality follows from \eqref{eq: first derivative of G}, \eqref{eq: first derivaitve of H}, Definition~\ref{def: linearized operator}, and the fact that the cofactor matrix of a symmetric matrix is symmetric. Furthermore, it follows from Step 1 that $d_{u} F(\phi,\Phi)$ is invertible; see \eqref{eq: isomorphism full problem}.

    Therefore, the implicit function theorem \cite{ambrosetti1995primer} ensures the existence of a function 
    \[
    	K\colon U\to V
    \]
    of class $C^{m+1}$, where $U\subset C^{m+2,\alpha}(\partial\Omega)$ is a neighborhood of $\phi$ and $V\subset C^{m+2,\alpha}(\overline{\Omega})$ a neighborhood of $\Phi$, such that
   \[
   		F(\eta,K(\eta))=0\quad\text{for all }\eta\in U.
   \]
   Finally, we may choose $\delta>0$ such that $B_{\delta}(\phi;C^{m+2,\alpha}(\partial\Omega))\subset U$. Hence, the solution operator
   \[
   		\mathcal{S}\vcentcolon = K|_{B_{\delta}(\phi;C^{m+2,\alpha}(\partial\Omega))}
   \]
  satisfies \ref{prop 1: solvability monge ampere}--\ref{prop 2: solution operator}.

  \medskip

  \textbf{Step 3: Linearization formulae.}

   First, let us notice that the implicit function theorem implies
   \[
   		d\mathcal{S}(\phi)h=\left.-(d_u F)^{-1}d_\eta F \right|_{(\phi,\mathcal{S}(\phi))}=(L_\Phi  ,\cdot|_{\partial\Omega})^{-1}(0,h)
   \]
  and hence $v_h= d\mathcal{S}(\phi)h$ solves
  \[
  	L_\Phi v_h=0\quad\text{in }\Omega\quad \text{and}\quad v_h|_{\partial\Omega}=h\quad\text{on }\partial\Omega.
  \]
  This establishes \ref{prop 3: linearization}.

	Since
	\[
	\mathcal S:
	C^{m+2,\alpha}(\partial\Omega)
	\to
	C^{m+2,\alpha}(\overline\Omega)
	\]
	is of class $C^{m+1}$ near $\phi$, the function $U_J$ is well-defined for
	$|J|\le m+1$.

	Let
	\[
	\eta(t)=\phi+\sum_{j\in J}t_j h_j,
	\qquad
	u(t)=\mathcal S(\eta(t)),
	\]
	for all $t=(t_j)_{j\in J}$ sufficiently small.
	Then
	\[
	\det D^2 u(t)=a(x,u(t),\nabla u(t))
	\quad\text{in }\Omega
	\]
	and
	\[
	u(t)|_{\partial\Omega}=\eta(t).
	\]

	We differentiate this identity with respect to all variables
	$t_j$ and evaluate at $t=0$. By the higher order chain rule, also known as Faà di Bruno's formula, we get
	\[
	\sum_{\pi\in\mathcal P(J)}
	d^{|\pi|}G(\Phi)[U_M]_{M\in\pi}
	=
	\sum_{\pi\in\mathcal P(J)}
	d^{|\pi|}H(\Phi)[U_M]_{M\in\pi},
	\]
	where the functions $G$ and $H$ are defined in \eqref{eq: defs of G and H}. The partition $\pi=\{J\}$ consists of a single block, and by definition of the $\Phi$-linearized operator $L_\Phi$ we have
	\[
	L_\Phi U_J=dG(\Phi)U_J-dH(\Phi)U_J.
	\]
	Hence
	\[
	L_\Phi U_J
	=
	-\sum_{\substack{\pi\in\mathcal P(J):\\ |\pi|\ge2}}
	d^{|\pi|}G(\Phi)[U_M]_{M\in\pi}
	+
	\sum_{\substack{\pi\in\mathcal P(J):\\ |\pi|\ge2}}
	d^{|\pi|}H(\Phi)[U_M]_{M\in\pi}.
	\]

	Next, differentiating the boundary identity
	\[
	u(t)|_{\partial\Omega}
	=
	\phi+\sum_{j\in J}t_j h_j
	\]
	with respect to all variables $t_j$, $j\in J$, we obtain
	\[
	U_J|_{\partial\Omega}=h_j
	\]
	if $J=\{j\}$, and
	\[
	U_J|_{\partial\Omega}=0
	\]
	if $|J|\ge2$.

	Finally, since
	\[
	G(u)=\det D^2u,
	\]
	we have
	\[
	d^r G(\Phi)[W_1,\ldots,W_r]
	=
	d^r(\det)(D^2\Phi)[D^2W_1,\ldots,D^2W_r].
	\]
	Likewise, since
	\[
	H(u)=a(\cdot,u,\nabla u),
	\]
	the higher derivatives are
	\[
	d^rH(\Phi)[W_1,\ldots,W_r]
	=
	d^r_{(z,p)}a(x,\Phi,\nabla\Phi)
	\big[
	(W_1,\nabla W_1),\ldots,(W_r,\nabla W_r)
	\big].
	\]
	Here, we use that $u\mapsto (u,\nabla u)$ is linear and hence Faà di Bruno's formula contains no additional lower-order terms.

	This finishes the proof of the theorem.
\end{proof}

\begin{remark}
	The assumption $a\in C^{2m+1,\alpha}$ is used only to ensure that
	\[
	H(u)=a(\cdot,u,\nabla u)
	\]
	defines a $C^{m+1}$-map
	\[
	C^{m+2,\alpha}(\overline\Omega)
	\to
	C^{m,\alpha}(\overline\Omega).
	\]
	Consequently, the solution operator $\mathcal S$ is of class $C^{m+1}$, allowing differentiation up to order $m+1$.
\end{remark}

\begin{remark}[Relation to the formal asymptotics]
The above theorem provides a rigorous justification of the formal expansion
\[
u_{\phi+\varepsilon\psi}=\Phi+\varepsilon v+o(\varepsilon),
\]
and shows that the linearized equation arises directly from differentiability of the solution map.
\end{remark}

\section{The nonlinear Cauchy data and its linearization}
\label{sec: nonlinear Cauchy}

\subsection{The nonlinear Cauchy data set}

We now introduce the boundary measurements associated with the nonlinear
Monge--Amp\`ere equation.

\begin{definition}[Nonlinear Cauchy data set]
	\label{def: nonlinear cauchy data}
	Let $\Omega \subset \R^n$ be a bounded domain with $C^{m+2,\alpha}$ boundary, $m \in \mathbb{N}_0$, $0<\alpha<1$, and $\phi\in C^{m+2,\alpha}(\partial\Omega)$. Furthermore, let
	\[
	a\in C^{2m+1,\alpha}(\overline\Omega\times\R\times\R^n)
	\]
	be positive. Suppose that $\Phi\in C^{m+2,\alpha}(\overline\Omega)$ is a uniformly convex solution of the nonlinear Monge--Amp\`ere equation \eqref{eq: Phi equation} with boundary value $\phi$ and let $\delta>0$ be the constant provided by
	Theorem~\ref{thm: local linearization MA}.

	The \emph{nonlinear Cauchy data set around}
	$(\phi,\partial_\nu\Phi|_{\partial\Omega})$ is defined by
	\[
	\mathscr C_a^{(\phi,\partial_\nu\Phi|_{\partial\Omega})}
	\vcentcolon =
	\Bigl\{
	(\eta,\partial_\nu u_\eta|_{\partial\Omega})
	\,;\,
	\eta\in B_\delta(\phi;C^{m+2,\alpha}(\partial\Omega))
	\Bigr\}
	\subset
	C^{m+2,\alpha}(\partial\Omega)
	\times
	C^{m+1,\alpha}(\partial\Omega),
	\]
	where $u_\eta \in C^{m+2,\alpha}(\overline{\Omega})$ denotes the unique solution of the nonlinear Monge--Amp\`ere equation \eqref{eq: MA well-posedness} with Dirichlet datum $\eta \in C^{m+2,\alpha}(\partial\Omega)$.
\end{definition}

\subsection{Linearization of the nonlinear Cauchy data set}

Next we study the linearization of the nonlinear Cauchy data set.

In the rest of this section, we suppose that the conditions of Theorem~\ref{thm: local linearization MA} hold. In particular, we assume that $a\in C^{m+2,\alpha}(\overline{\Omega},\R,\R^n)$ is positive, $\phi\in C^{m+2,\alpha}(\partial\Omega)$, and $\Phi\in C^{m+2,\alpha}(\overline{\Omega})$ is a uniformly convex solution of the problem 
\begin{equation}
\begin{cases}
\det D^2 \Phi = a(x,\Phi,\nabla \Phi) & \text{in } \Omega, \\
\Phi = \phi & \text{on } \partial\Omega.
\end{cases}
\end{equation}
Throughout this section, for any $\eta\in B_{\delta}(\phi;C^{m+2,\alpha}(\partial\Omega))$, we denote by $u_
\eta\in C^{m+2,\alpha}(\overline\Omega)$ the unique solution of the nonlinear Monge--Amp\`ere problem 
\begin{equation}
\begin{cases}
\det D^2 u = a(x,u,\nabla u) & \text{in } \Omega, \\
u = \eta & \text{on } \partial\Omega
\end{cases}
\end{equation}
and, for any $h\in C^{m+2,\alpha}(\partial\Omega)$, by
$v_h \in C^{m+2,\alpha}(\overline\Omega)$ the unique solution of the linearized problem
\[
\begin{cases}
L_\Phi v=0 & \text{in }\Omega,\\
v=h & \text{on }\partial\Omega,
\end{cases}
\]
where $L_\Phi$ is the (uniformly elliptic) $\Phi$-linearized operator
\begin{equation}
\label{eq: L Phi operator}
 L_\Phi v \vcentcolon = \tr(\mathcal{C}_\Phi D^2 v)
- b_\Phi \cdot \nabla v - q_\Phi v,\quad v\in C^2(\Omega),
\end{equation}
whose $C^{m,\alpha}$-- coefficients are given by
\[
	\mathcal{C}_\Phi \vcentcolon = \cof D^2\Phi,\quad b_\Phi \vcentcolon = \nabla_p a(x,\Phi,\nabla \Phi),\quad\text{and}\quad q_\Phi \vcentcolon = \partial_z a(x,\Phi,\nabla \Phi)\geq 0.
\]
The next assertion studies the first variation of the nonlinear Cauchy data set.

\begin{proposition}[Linearization of the nonlinear Cauchy data]
	\label{prop: linearization cauchy data}

	For any $h\in C^{m+2,\alpha}(\partial\Omega)$, we have
	\[
	u_{\phi+t h}
	=
	\Phi+t v_h+o(t)
	\quad\text{in }C^{m+2,\alpha}(\overline\Omega)
	\]
	as $t\to 0$.

	Consequently,
	\[
	\partial_\nu u_{\phi+t h}
	=
	\partial_\nu\Phi
	+
	t\,\partial_\nu v_h
	+
	o(t)
	\quad\text{in }C^{m+1,\alpha}(\partial\Omega)
	\]	
	as $t\to 0$.

	In particular, the first variation of the nonlinear Cauchy data set $\mathscr C_a^{(\phi,\partial_\nu\Phi|_{\partial\Omega})}$ coincides with the Cauchy data set of the linearized operator
	$L_\Phi$,
	\begin{equation}
	\label{eq: linear Cauchy set}
		\mathscr{C}^{\textrm{lin}}_{a,\Phi}\vcentcolon =\Bigl\{
		(\eta,\partial_\nu v_\eta|_{\partial\Omega})
		\,;\,
		\eta\in C^{m+2,\alpha}(\partial\Omega)
		\Bigr\}
		\subset
		C^{m+2,\alpha}(\partial\Omega)
		\times
		C^{m+1,\alpha}(\partial\Omega),
	\end{equation}
\end{proposition}

\begin{proof}
	The expansion
	\[
	u_{\phi+t h}
	=
	\Phi+t v_h+o(t)
	\]
	follows from Theorem~\ref{thm: local linearization MA}, \ref{prop 2: solution operator}.
	Applying the continuous trace operator
	\[
	u\mapsto \partial_\nu u |_{\partial\Omega}
	\]
	from
	\[
	C^{m+2,\alpha}(\overline\Omega)
	\to
	C^{m+1,\alpha}(\partial\Omega)
	\]
	yields
	\[
	\partial_\nu u_{\phi+t h}
	=
	\partial_\nu\Phi+t\partial_\nu v_h+o(t)\quad\text{on }\partial\Omega.
	\]
	Consequently, the result follows.
\end{proof}

Since by Remark~\ref{rem: divergence form}, $L_\Phi$ is a divergence-form elliptic operator, we next establish the relation of the linear Cauchy data set to the (conormal) Dirichlet-to-Neumann map.

\begin{proposition}[Cauchy data determine the (conormal) DN map]
	\label{prop: cauchy determines conormal DN}
	The linear Cauchy data set $\mathscr{C}^{\textrm{lin}}_{a,\Phi}$ associated to $\Phi$ determines the \emph{Dirichlet-to-Neumann (DN) map}
	\begin{equation}
	\label{eq: strong form DN}
	\Lambda_{\Phi} \eta
	\vcentcolon =
	\nu\cdot \mathcal C_\Phi\nabla v_\eta |_{\partial\Omega}
	\end{equation}
	for all $\eta\in C^{m+2,\alpha}(\partial\Omega)$, where $v_\eta$ is the unique solution of
	\begin{equation}
	\label{eq: linearized eq DN}
	L_\Phi v_\eta=0,\qquad v_\eta|_{\partial\Omega}=\eta.
	\end{equation}
	More precisely, if
	\[
		(\eta,\partial_\nu v_\eta |_{\partial\Omega})\in \mathscr{C}^{\textrm{lin}}_{a,\Phi},
	\]
	then
	\[
	\Lambda_{\Phi} \eta
	=
	\nu\cdot \mathcal C_\Phi\nabla_{\tan} \eta
	+
	(\partial_\nu v_\eta |_{\partial\Omega})\,\nu\cdot \mathcal C_\Phi\nu.
	\]
	Moreover, by density, $\mathscr{C}^{\textrm{lin}}_{a,\Phi}$ determines $\Lambda_{\Phi} \eta$ for all $\eta\in H^{1/2}(\partial\Omega)$ and in this case it is given by the weak form
	\begin{equation}
	\label{eq: weak form}
		\langle \Lambda_{\Phi} \eta,\psi\rangle=\int_{\Omega}\left(\mathcal{C}_\Phi \nabla v_\eta \cdot \nabla \Psi+b_\Phi \cdot  \nabla v_\eta \Psi+q_\Phi v_\eta \Psi\right)\,dx,
	\end{equation}
	where $v_\eta\in H^1(\Omega)$ is the unique solution of \eqref{eq: linearized eq DN} and $\Psi$ and $H^1(\Omega)$-extension of $\psi\in H^{1/2}(\partial\Omega)$. Furthermore, it is continuous as a map
	\[
		\Lambda_{\Phi}\colon H^{1/2}(\partial\Omega)\to H^{-1/2}(\partial\Omega).
	\]
\end{proposition}

\begin{proof}
	We prove the result in four steps.

	\medskip

	\textbf{Step 1: Smooth boundary data.} 
	We first establish the assertions for regular Dirichlet data. If $\eta\in C^{m+2,\alpha}(\partial\Omega)$, then $v_\eta \in C^{m+2,\alpha}(\overline{\Omega})$ and we can decompose the gradient on $\partial\Omega$ as
	\[
	\nabla v_\eta
	=
	\nabla_{\tan} v_\eta+(\partial_{\nu} v_\eta) \nu=\nabla_{\tan} \eta+(\partial_{\nu} v_\eta) \nu.
	\]
	Therefore,
	\[
	\nu\cdot\mathcal C_\Phi\nabla v_\eta
	=
	\nu\cdot\mathcal C_\Phi\nabla_{\tan}\eta 
	+
	(\partial_{\nu} v_\eta)\,\nu\cdot\mathcal C_\Phi\nu.
	\]
	Since $\Phi$ is fixed, we conclude that the linear Cauchy data  $\mathscr{C}^{\textrm{lin}}_{a,\Phi}$ set determines $\Lambda_{\Phi}$. 

	\medskip

	\textbf{Step 2: Weak form is well-defined.} 

	Next, suppose that $\eta\in H^{1/2}(\partial\Omega)$. Since $L_\Phi$ has $C^{m,\alpha}$--coefficients and $q_\Phi\geq 0$, \cite[Theorem~8.3]{GT} guarantees the well-posedness of \eqref{eq: linearized eq DN} in $H^1(\Omega)$ for any $\eta\in H^{1/2}(\partial\Omega)$. Furthermore, by \cite[Corollary~8.7]{GT} the solution $v_\eta\in H^1(\Omega)$ depends continuously on the boundary data $\eta\in H^{1/2}(\partial\Omega)$, that is
	\begin{equation}
	\label{eq: continuity estimate}
		\|v_\eta\|_{H^{1}(\Omega)}\leq C\|\eta\|_{H^{1/2}(\partial\Omega)}
	\end{equation}
	for some $C>0$. Hence, \eqref{eq: weak form} is well-defined for all $\eta \in H^{1/2}(\partial\Omega)$.

	\textbf{Step 3: Weak equals strong for $C^{m+2,\alpha}(\partial\Omega)$ data.} Let $\eta\in C^{m+2,\alpha}(\partial\Omega)$. If $\Phi\in C^{3,\alpha}(\overline{\Omega})$, then the equality of \eqref{eq: strong form DN} and \eqref{eq: weak form} follows by a simple integration by parts argument. More precisely,
	\[
		\int_{\partial\Omega}(\Lambda_{\Phi}\eta)\psi\,d\mathcal{H}^{n-1}=\int_{\Omega}\left(\mathcal{C}_\Phi \nabla v_\eta \cdot \nabla \Psi+b_\Phi \cdot \nabla v_\eta \Psi+q_\Phi v_\eta \Psi\right)\,dx.
	\]
	 If $\Phi\in C^{2,\alpha}(\overline{\Omega})$ we first extend it to a function $C^{2,\alpha}(\R^n)$ with compact support and then mollify it to get a function 
	\[
		\Phi_\eps=\rho_\eps \ast \Phi \in C^{\infty}_b(\R^n),
	\]
	where $(\rho_\eps)_{\eps>0}$ are the standard mollifiers in $\R^n$. Using Piola's identity, this yields
	\[
	\begin{split}
			\int_{\partial\Omega}(\Lambda_{\Phi}\eta)\psi\,d\mathcal{H}^{n-1}&=	\int_{\partial\Omega}\nu\cdot \mathcal C_\Phi\nabla v_\eta |_{\partial\Omega}\psi\,d\mathcal{H}^{n-1}\\
			&=\lim_{\eps\to 0}\int_{\partial\Omega}\nu\cdot \mathcal C_{\Phi_\eps}\nabla v_\eta |_{\partial\Omega}\psi\,d\mathcal{H}^{n-1}\\
			&=\lim_{\eps\to 0}\int_{\Omega}\Div\left( \mathcal C_{\Phi_\eps}\nabla v_\eta \Psi\right)\,dx\\
			&=\lim_{\eps\to 0}\left(\int_{\Omega}\Div( \mathcal C_{\Phi_\eps}\nabla v_\eta) \Psi\,dx+\int_{\Omega}\mathcal{C}_{\Phi_\eps} \nabla v_\eta \cdot \nabla \Psi\,dx\right)\\
			&=\lim_{\eps\to 0}\left(\int_{\Omega}\tr( \mathcal C_{\Phi_\eps}D^2 v_\eta) \Psi\,dx+\int_{\Omega}\mathcal{C}_{\Phi_\eps} \nabla v_\eta \cdot \nabla \Psi\,dx\right)\\
			&=\int_{\Omega}\left(\mathcal{C}_\Phi \nabla v_\eta \cdot \nabla \Psi+b_\Phi\cdot  \nabla v_\eta \Psi+q_\Phi v_\eta \Psi\right)\,dx.
	\end{split}
	\]
	Hence, also in this case the strong and weak forms coincide.

	\medskip

		\textbf{Step 4: $\mathscr{C}_{a,\Phi}^{\textrm{lin}}$ determines $\Lambda_{\Phi}$ on $H^{1/2}(\partial\Omega)$.} Finally, we may use density of $C^{m+2,\alpha}(\partial\Omega)$ in $H^{1/2}(\partial\Omega)$, the continuity estimate \eqref{eq: continuity estimate}, Step 3, and Step 1 to conclude that $\mathscr{C}_{a,\Phi}^{\textrm{lin}}$ determines $\Lambda_{\Phi}$ on $H^{1/2}(\partial\Omega)$. The continuity of $\Lambda_\Phi$ as a map from $H^{1/2}(\partial\Omega)$ to its dual $H^{-1/2}(\partial\Omega)=(H^{1/2}(\partial\Omega))^*$ is also a consequence of \eqref{eq: continuity estimate}.
\end{proof}

\subsection{Higher-order linearization of the nonlinear Cauchy data set}

	\begin{proposition}[Higher-order linearization of the nonlinear Cauchy data]
		\label{prop: higher order linearization cauchy data}
		Assume the hypotheses of Theorem~\ref{thm: local linearization MA}. Let
		$1\le k\le m+1$, and let
		\[
		h_1,\ldots,h_k\in C^{m+2,\alpha}(\partial\Omega).
		\]
		For every nonempty subset $J\subset \{1,\ldots,k\}$, set
		\[
		U_J
		:=
		d^{|J|}\mathcal S(\phi)[h_j]_{j\in J}
		\in C^{m+2,\alpha}(\overline\Omega);
		\]
		see Theorem~\ref{thm: local linearization MA}, \ref{prop 4: higher order linearization}.
		Then, for
		\[
		\eta(t)
		=
		\phi+\sum_{j=1}^k t_jh_j,
		\qquad
		u(t):=u_{\eta(t)},
		\]
		one has
		\[
		u(t)
		=
		\Phi
		+
		\sum_{\emptyset\neq J\subset \{1,\ldots,k\}}
		\frac{t_J}{|J|!}\,U_J
		+
		o(|t|^k)
		\quad\text{in }C^{m+2,\alpha}(\overline\Omega),
		\]
		where $t_J:=\prod_{j\in J}t_j$.

		Consequently,
		\[
		\partial_\nu u(t)
		=
		\partial_\nu\Phi
		+
		\sum_{\emptyset\neq J\subset \{1,\ldots,k\}}
		\frac{t_J}{|J|!}\,\partial_\nu U_J
		+
		o(|t|^k)
		\quad\text{in }C^{m+1,\alpha}(\partial\Omega).
		\]

		Equivalently, the $r$-th variation of the nonlinear Cauchy data set at the
		background $\Phi$ is given by
		\[
		d^r
		\Bigl(
		\eta\mapsto
		(\eta,\partial_\nu u_\eta|_{\partial\Omega})
		\Bigr)_{\eta=\phi}
		[h_1,\ldots,h_r]
		=
		\begin{cases}
			(h_1,\partial_\nu U_{{1}}|_{\partial\Omega}),
			& r=1,\\
			(0,\partial_\nu U_{{1,\ldots,r}}|_{\partial\Omega}),
			& 2\le r\le m+1.
		\end{cases}
		\]
	\end{proposition}

	\begin{proof}
		By Theorem~\ref{thm: local linearization MA}, the solution map
		\[
		\mathcal S:
		C^{m+2,\alpha}(\partial\Omega)
		\to
		C^{m+2,\alpha}(\overline\Omega)
		\]
		is of class $C^{m+1}$ near $\phi$. Hence, for
		\[
		\eta(t)=\phi+\sum_{j=1}^k t_jh_j,
		\]
		Taylor expansion in Banach spaces gives
		\[
		u(t)
		=\mathcal S(\eta(t))	=
		\Phi
		+
		\sum_{\emptyset\neq J\subset \{1,\ldots,k\}}
		\frac{t_J}{|J|!}
		d^{|J|}\mathcal S(\phi)[h_j]_{j\in J}
		+
		o(|t|^k)
		\]
		in $C^{m+2,\alpha}(\overline\Omega)$. By definition,
		\[
		U_J=d^{|J|}\mathcal S(\phi)[h_j]_{j\in J}.
		\]
		Since the normal trace map
		\[
		u\mapsto \partial_\nu u|_{\partial\Omega}
		\]
		is continuous from $C^{m+2,\alpha}(\overline\Omega)$ to $C^{m+1,\alpha}(\partial\Omega)$,
		we may apply it to the Taylor expansion and obtain
		\[
		\partial_\nu u(t)
		=
		\partial_\nu\Phi
		+
		\sum_{\emptyset\neq J\subset{1,\ldots,k}}
		\frac{t_J}{|J|!}\partial_\nu U_J
		+
		o(|t|^k)
		\]
		in $C^{m+1,\alpha}(\partial\Omega)$.

		Finally, the boundary value identity
		\[
		u(t)|_{\partial\Omega}
		=
		\phi+\sum_{j=1}^k t_jh_j
		\]
		implies that the first variation of the first Cauchy component is $h_j$, while
		all higher variations of the first component vanish. Therefore,
		\[
		d^r
		\Bigl(
		\eta\mapsto
		(\eta,\partial_\nu u_\eta|_{\partial\Omega})
		\Bigr)_{\eta=\phi}
		[h_1,\ldots,h_r]
		=
		\begin{cases}
			(h_1,\partial_\nu U_{{1}}|_{\partial\Omega}),
			& r=1,[0.4em]\\
			(0,\partial_\nu U_{{1,\ldots,r}}|_{\partial\Omega}),
			& 2\le r\le m+1.
		\end{cases}
		\]
		This proves the claim.
	\end{proof}

\section{Alessandrini identity}
\label{sec: Alessandrini identity}

Our next goal is to establish an Alessandrini identity for the linear problem. 

Throughout this section, we write $(L_\Phi)^*$ for the $L^2$-adjoint of $L_\Phi$, which is (strongly) given by
\[
L_\Phi^*\rho=\tr(\mathcal{C}_\Phi D^2\rho)+\Div(b_\Phi \rho )-q_\Phi \rho
\]
for smooth functions $\rho\colon \Omega\to \R$. Similarly, as in the proof of Proposition~\ref{prop: cauchy determines conormal DN}, one may deduce from $(\mathcal{C}_\Phi,b_\Phi,q_\Phi)\in C^{m,\alpha}(\overline{\Omega})$ and $q_\Phi \geq 0$ that the Dirichlet problem 
	\begin{equation}
\label{eq: adjoint problem}
\begin{cases}
L^*_\Phi w= 0 & \text{in } \Omega, \\
w= \eta & \text{on } \partial\Omega.
\end{cases}
\end{equation}
is well-posed in $H^1(\Omega)$ for any Dirichlet datum $\eta\in H^{1/2}(\partial\Omega)$ and the solution map
\[
	H^{1/2}(\partial\Omega)\ni \eta\mapsto w_\eta \in H^1(\Omega)
\]
is continuous (see \cite[Theorem~8.3 \& Corollary~8.7]{GT}). Hence, we may define the related (adjoint) Dirichlet-to-Neumann map $\Lambda^*_\Phi$ by
\begin{equation}
\label{eq: adjoint DN map}
	\langle \Lambda^*_\Phi \eta,\psi\rangle \vcentcolon = \int_{\Omega}\left(\mathcal{C}_\Phi\nabla w_\eta\cdot\nabla\Psi+w_\eta (b_\Phi\cdot \nabla \Psi)+q_\Phi w_\eta \Psi\right)\,dx
\end{equation}
for all $\eta,\psi\in H^{1/2}(\partial\Omega)$, where $w_\eta\in H^1(\Omega)$ is the unique solution of \eqref{eq: adjoint problem} and $\Psi$ any $H^1(\Omega)$-extension of $\psi$.

By letting $\Psi=v_\psi\in H^1(\Omega)$, where $v_\psi$ is the unique $H^1$-solution of 
	\begin{equation}
\begin{cases}
L_\Phi v= 0 & \text{in } \Omega, \\
v= \psi & \text{on } \partial\Omega,
\end{cases}
\end{equation}
we deduce that
\begin{equation}
\label{eq: relation adjoint DN}
		\langle \Lambda^*_\Phi \eta,\psi\rangle=	\langle \Lambda_\Phi \psi,\eta\rangle
\end{equation}
for all $\eta,\psi\in H^{1/2}(\partial\Omega)$; see equation~\eqref{eq: weak form}.

Now, we can prove the following Alessandrini identity.

\begin{proposition}[Alessandrini identity]
\label{prop: Alessandrini-short}
	Assume the hypotheses of Theorem~\ref{thm: local linearization MA}. Suppose that $\Phi_j\in C^{m+2,\alpha}(\overline{\Omega})$ is a uniformly convex background solution of the nonlinear Monge--Amp\`ere equation
	\[
	\begin{cases}
			\det D^2 \Phi=a_j(x,\Phi,\nabla \Phi)\quad&\text{in }\Omega,\\
			\Phi=\phi\quad&\text{on }\partial\Omega,
	\end{cases}
	\]
	for $j=1,2$. Furthermore, let $L^j_{\Phi_j}$ be the $\Phi_j$-linearized operator with coefficients $(\mathcal{C}_{\Phi_j}, b^j_{\Phi_j},q^j_{\Phi_j})$ such that $q^j_{\Phi_j}\geq 0$ and $\Lambda^j_{\Phi_j}$ the associated DN map (see Definition~\ref{def: linearized operator} and Proposition~\ref{prop: cauchy determines conormal DN}).

Let $v,w \in H^1(\Omega)$ be weak solutions of 
\[
L^1_{\Phi_1} v=0 \quad \text{in }\Omega, \qquad
(L^2_{\Phi_2})^* w=0 \quad \text{in }\Omega,
\]
with boundary values $\psi_1$ and $\psi_2$, respectively. 

Then
\[
\begin{split}
\langle (\Lambda^1_{\Phi_1} -\Lambda^2_{\Phi_2} )\psi_1,\psi_2\rangle
&=\int_\Omega
(\mathcal C_{\Phi_1}-\mathcal C_{\Phi_2})
\nabla v\cdot\nabla w\,dx\\
&+
\int_\Omega (q^1_{\Phi_1} -q^2_{\Phi_2} )v w\,dx
+\int_\Omega (b^1_{\Phi_1}-b^2_{\Phi_2})\cdot \nabla v\, w\,dx.
\end{split}
\]
\end{proposition}

\begin{proof}
By definition of the linearized DN map $\Lambda^1_{\Phi_1} $, the adjoint map $(\Lambda^2_{\Phi_2})^*$, and \eqref{eq: relation adjoint DN} we have
\[
\begin{split}
		\langle (\Lambda^1_{\Phi_1}-\Lambda^2_{\Phi_2}) \psi_1,\psi_2\rangle&=\langle \Lambda^1_{\Phi_1} \psi_1,\psi_2\rangle-\langle \Lambda^2_{\Phi_2} \psi_1,\psi_2\rangle\\
	&=\langle \Lambda^1_{\Phi_1} \psi_1,\psi_2\rangle-\langle (\Lambda^2_{\Phi_2})^* \psi_2,\psi_1\rangle\\
	&=\int_\Omega
	(\mathcal C_{\Phi_1}-\mathcal C_{\Phi_2})
	\nabla v\cdot\nabla w\,dx\\
	&\quad
	+
	\int_\Omega (b^1_{\Phi_1} -b^2_{\Phi_2} )\cdot \nabla v\,w\,dx+\int_\Omega (q^1_{\Phi_1}-q^2_{\Phi_2}) v w\,dx.
\end{split}
\]
This finishes the proof of the claim.
\end{proof}

\section{Reduction principle for first-order derivatives of the nonlinearity}
\label{sec: reduction linear ip for gradient}

\begin{theorem}[Reduction principle for first order derivatives]
	\label{thm: reduction-principle-full}
	Let $m\in\N_0$, $0<\alpha<1$, and $a_1,a_2\in C^{2m+1,\alpha}(\overline{\Omega}\times\R\times\R^n)$
	be positive nonlinearities. Suppose that $\Phi_j\in C^{m+2,\alpha}(\overline{\Omega})$ is a uniformly convex background solution of the nonlinear Monge--Amp\`ere equation
	\[
	\begin{cases}
	\det D^2 \Phi=a_j(x,\Phi,\nabla \Phi)\quad&\text{in }\Omega,\\
	\Phi=\phi\quad&\text{on }\partial\Omega,
	\end{cases}
	\]
	for $j=1,2$ and $\phi\in C^{m+2,\alpha}(\partial\Omega)$. Furthermore, let $L^j_{\Phi_j}$ be the $\Phi_j$-linearized operator with coefficients $(\mathcal{C}_{\Phi_j}, b^j_{\Phi_j},q^j_{\Phi_j})$ such that $q^j_{\Phi_j}\geq 0$. 

	Assume that
	\begin{equation}
	\label{eq: compatibility condition a}
	\begin{split}
		\partial_\nu \Phi_1|_{\partial\Omega}=\partial_\nu \Phi_2|_{\partial\Omega}
	\end{split}
	\end{equation}
	and 
	\begin{equation}
	\label{eq: compatibility condition b}
			a_1(x,\phi,\nabla_{\tan}\phi+(\partial_\nu \Phi_1 |_{\partial\Omega})\nu)=a_2(x,\phi,\nabla_{\tan}\phi+(\partial_\nu \Phi_2 |_{\partial\Omega})\nu).
	\end{equation}
	for all $x\in\partial\Omega$.

	If the nonlinear Cauchy data sets around $(\phi,\partial_\nu \Phi_j|_{\partial\Omega})$ coincide:
	\[
	\mathscr C_{a_1}^{(\phi,\partial_\nu\Phi_1|_{\partial\Omega})}
	=
	\mathscr C_{a_2}^{(\phi,\partial_\nu\Phi_2|_{\partial\Omega})},
	\]
	then one has
	\[
	\begin{split}
			0=\langle (\Lambda^1_{\Phi_1}-\Lambda^2_{\Phi_2})\psi_1,\psi_2\rangle&=\int_\Omega
			(\mathcal C_{\Phi_1}-\mathcal C_{\Phi_2})
			\nabla v\cdot\nabla w\,dx\\
			&+\int_\Omega
		(q_{\Phi_1}^1-q_{\Phi_2}^2)vw\,dx
		+
		\int_\Omega
		(b_{\Phi_1}^1-b_{\Phi_2}^2)\cdot\nabla v\,w\,dx
	\end{split}
	\]
	for all $\psi_1,\psi_2\in H^{1/2}(\partial\Omega)$, where $v,w\in H^1(\Omega)$ are the unique solutions of 
	\[
	\begin{cases}
			L^1_{\Phi_1} v=0 & \quad \text{in }\Omega,\\
		v=\psi_1 & \quad \text{on }\partial\Omega.
	\end{cases}
	 \quad\text{and}\quad
	 \begin{cases}
		(L^2_{\Phi_2})^* w=0 & \quad \text{in }\Omega,\\
		 w=\psi_2 & \quad \text{on }\partial\Omega.
	\end{cases}
	\]
\end{theorem}

\begin{proof}
	By Proposition~\ref{prop: linearization cauchy data}, the first variation of
	$\mathscr C_{a_j}^{\Phi_j}$ at the background solution $\Phi_j$ is precisely
	the linear Cauchy data set $\mathscr C_{a_j,\Phi_j}^{\mathrm{lin}}$. Since
	\[
		\mathscr C_{a_1}^{(\phi,\partial_\nu\Phi_1|_{\partial\Omega})}
	=
	\mathscr C_{a_2}^{(\phi,\partial_\nu\Phi_2|_{\partial\Omega})},
	\]
	their first variations coincide, and therefore
	\[
	\mathscr C_{a_1,\Phi_1}^{\mathrm{lin}}
	=
	\mathscr C_{a_2,\Phi_2}^{\mathrm{lin}}.
	\]
	Next, we assert that the compatibility conditions \eqref{eq: compatibility condition a}--\eqref{eq: compatibility condition b} imply
	\begin{equation}
	\label{eq: compatibility 2}
	\mathcal{C}_{\Phi_1}|_{\partial\Omega}=\mathcal{C}_{\Phi_2}|_{\partial\Omega}.
	\end{equation}

	To see this, let $x_0\in \partial\Omega$, and choose a local orthonormal frame on $\partial\Omega$ near $x_0$,
	\[
	\tau_1,\ldots,\tau_{n-1}.
	\]
	We use the convention
	\[
	\mathrm{II}(X,Y):=(D_X\nu)\cdot Y,
	\]
	where $\nu\colon \partial\Omega\to \mathbb{S}^{n-1}$ denotes the exterior unit normal and $D_X$ is the directional derivative given by
	\[
		D_X Y=(X\cdot \nabla)Y.
	\]
	Let $\nabla_X^{\partial\Omega}$ be the Levi--Civita connection on $\partial\Omega$ inherited from $\R^n$, that is,
	\[
		\nabla_X^{\partial\Omega}Y=D_X Y-((D_X Y)\cdot \nu)\nu.
	\]
	If $Y$ is tangential, then $Y\cdot \nu=0$ on $\partial\Omega$  and hence
	\[
	0=X(Y\cdot \nu)=(D_X Y)\cdot \nu+Y\cdot D_X \nu\quad\text{on }\partial\Omega.
	\]
	Thus, Gauss' formula reads
	\[
		D_X Y=\nabla_X^{\partial\Omega}Y-\mathrm{II}(X,Y)\nu.
	\]
	Let $j=1,2$ and set
	\[
	\psi_j:=\partial_\nu \Phi_j |_{\partial\Omega}.
	\]
	Then, on $\partial\Omega$, the full gradient is given by
	\begin{equation}
	\label{eq: decomposition of gradient}
	\nabla \Phi_j =	\nabla_{\tan}\Phi_j+(\partial_\nu \Phi_j)\nu
	=
	\nabla_{\tan}\phi+\psi_j \nu.
	\end{equation}
	The intrinsic Hessian on the boundary is
	\[
		\nabla_{\partial\Omega}^2 \Phi_j (X,Y)=X(Y\Phi_j )-(\nabla_X^{\partial\Omega} Y)\Phi_j .
	\]
	Using Gauss' formula, we see that the tangential-tangential components of the Hessian $D^2 \Phi$ satisfy
	\[
	\begin{split}
		D^2 \Phi_j(X,Y)&= X(Y\Phi_j )-(D_X Y)\Phi_j \\
		&=X(Y\Phi_j )-(\nabla_X^{\partial\Omega}Y)\Phi_j +\mathrm{II}(X,Y)\partial_\nu \Phi_j \\
		&=\nabla_{\partial\Omega}^2 \Phi_j (X,Y)+\mathrm{II}(X,Y)\partial_\nu \Phi_j
	\end{split}
	\]
	for all tangential vector fields $X,Y$. In particular,
	\[
	D^2 \Phi_j (\tau_i,\tau_k)
	=
	\nabla_{\partial\Omega}^2\phi(\tau_i,\tau_k)
	+
	\psi_j\,\mathrm{II}_{ik}.
	\]
	The mixed tangential--normal components are
	\[
	\begin{split}
		D^2\Phi_j (X,\nu)
		&= X(\partial_\nu \Phi_j )-D_X \nu\cdot \nabla \Phi_j =X\psi_j - D_X\nu\cdot \nabla_{\tan}\phi		,
	\end{split}
	\]
	because $\nabla \Phi_j =\nabla_{\textrm{tan}}\phi+\psi_j \nu$, with $\nabla_{\textrm{tan}}f=(\nabla_{\tau_j}f)\tau_j$, and $D_X\nu$ is tangent.
	Equivalently, using $\mathrm{II}(X,Y)=D_X\nu\cdot Y$,
	\[
	D^2\Phi_j (\tau_i,\nu)
	=
	\tau_i \psi_j-\sum_{k=1}^{n-1}\mathrm{II}_{ik}\,\tau_k\phi.
	\]
	Thus, with respect to the orthonormal frame
	\[
	\tau_1,\ldots,\tau_{n-1},\nu,
	\]
	the Hessian has the block form
	\[
	D^2\Phi_j |_{\partial\Omega}
	=
	\begin{pmatrix}
	A^{j} & B^{j}\\
	(B^{j})^T & c^{j}
	\end{pmatrix},
	\]
	where
	\[
	A^{j}_{ik}
	=
	\nabla_{\partial\Omega}^2\phi(\tau_i,\tau_k)
	+
	\psi_j \,\mathrm{II}_{ik},
	\]
	\[
	B^{j}_i
	=
	\tau_i \psi_j-\sum_{k=1}^{n-1}\mathrm{II}_{ik}\,\tau_k\phi,
	\]
	and
	\[
	c^{j}=
	D^2\Phi_j (\nu,\nu)
	=
	\partial_\nu^2 \Phi_j
	\]
	is the only second-order boundary component not determined directly by the first-order boundary jet
	\[
	(\Phi_j |_{\partial\Omega},\partial_\nu \Phi_j|_{\partial\Omega}).
	\]
	Since $\Phi_j$ solves
	\[
	\det D^2\Phi_j
	=
	a_j(x,\Phi_j,\nabla\Phi_j),
	\]
	then, on $\partial\Omega$,
	\[
	\det
	\begin{pmatrix}
	A^{j} & B^{j}\\
	(B^{j})^T & c^{j}
	\end{pmatrix}
	=
	a_j(x,\phi,\nabla_{\tan}\phi+(\partial_\nu \Phi_j|_{\partial\Omega})\nu).
	\]
	Since $D^2\Phi_j$ is positive definite, the tangential block $A^{j}$ is positive
	definite. Indeed, for $\xi\neq 0$, we have
	\[
		\xi\cdot A^j\xi=(\xi,0)\cdot D^2\Phi_j (\xi,0)>0.
	\]
	Hence
	\[
	\det
	\begin{pmatrix}
	A^{j} & B^{j}\\
	(B^{j})^T & c^{j}
	\end{pmatrix}
	=
	(\det A^{j})\,(c^{j}-(B^{j})^T(A^{j})^{-1}B^{j}),
	\]
	and therefore
	\[
	\partial_\nu^2\Phi_j
	=
	(B^{j})^T(A^{j})^{-1}B^{j}
	+
	\frac{a_j(x,\phi,\nabla_{\tan}\phi+(\partial_\nu \Phi_j |_{\partial\Omega})\nu)}{\det A^{j}}
	\qquad\text{on }\partial\Omega.
	\]
	Consequently, conditions
	\eqref{eq: compatibility condition a}
	and
	\eqref{eq: compatibility condition b}
	imply that the matrices $A^{j}$, $B^{j}$, and $c^{j}$ coincide for
	$j=1,2$. Hence
	\[
	D^2\Phi_1|_{\partial\Omega}
	=
	D^2\Phi_2|_{\partial\Omega}.
	\]
	Therefore
	\[
	\mathcal C_{\Phi_1}|_{\partial\Omega}
	=
	\cof D^2\Phi_1|_{\partial\Omega}
	=
	\cof D^2\Phi_2|_{\partial\Omega}
	=
	\mathcal C_{\Phi_2}|_{\partial\Omega},
	\]
	which establishes \eqref{eq: compatibility 2}.
	Hence, we have shown that \eqref{eq: compatibility 2} follows from	\eqref{eq: compatibility condition a}--\eqref{eq: compatibility condition b}.

	Now, since the linear Cauchy data set coincide and the cofactor matrices satisfy the compatibility condition \eqref{eq: compatibility 2}, the proof of Proposition~\ref{prop: cauchy determines conormal DN} implies
	\[
	\Lambda_{\Phi_1}^1
	=
	\Lambda_{\Phi_2}^2 \quad \text{in }L(H^{1/2}(\partial\Omega),H^{-1/2}(\partial\Omega)).
	\]

	The desired integral identity now follows from the Alessandrini identity,
	Proposition~\ref{prop: Alessandrini-short}.
\end{proof}

From the proof, we deduce the following remark.

\begin{remark}[Boundary determination of the cofactor matrix]
	\label{rem: boundary determination cofactor}
	The compatibility condition
	\[
	\mathcal C_{\Phi_1}|_{\partial\Omega}
	=
	\mathcal C_{\Phi_2}|_{\partial\Omega}
	\]
	used in the proof of Theorem~\ref{thm: reduction-principle-full} is not an independent assumption. Indeed, suppose that
	\[
	\Phi_j\in C^{m+2,\alpha}(\overline{\Omega}),
	\qquad j=1,2,
	\]
	are uniformly convex solutions of
	\[
	\det D^2\Phi_j
	=
	a_j(x,\Phi_j,\nabla\Phi_j)
	\quad\text{in }\Omega,
	\]
	with common boundary value $\phi$. Assume furthermore that
	\[
	\partial_\nu\Phi_1|_{\partial\Omega}
	=
	\partial_\nu\Phi_2|_{\partial\Omega}
	\]
	and
	\[
	a_1(x,\phi,\nabla_{\tan}\phi+(\partial_\nu\Phi_1|_{\partial\Omega})\nu)
	=
	a_2(x,\phi,\nabla_{\tan}\phi+(\partial_\nu\Phi_1|_{\partial\Omega})\nu)
	\]
	for all $x\in\partial\Omega$. Then
	\[
	D^2\Phi_1|_{\partial\Omega}
	=
	D^2\Phi_2|_{\partial\Omega}\quad\text{and}\quad\mathcal C_{\Phi_1}|_{\partial\Omega}
	=
	\mathcal C_{\Phi_2}|_{\partial\Omega}.
	\]
\end{remark}

\section{Higher-order integral identities for fixed background solution}
\label{sec: higher order integral id fixed background}

\begin{proposition}[Higher-order integral identity]
	\label{prop: higher order Alessandrini}
	Assume the hypotheses of Theorem~\ref{thm: local linearization MA}. Suppose that $a_1,a_2$ admit the same uniformly convex background solution
	$\Phi$, and suppose that the first linearizations of $a_j$ along the background jet $(\Phi,\nabla \Phi)$ agree, so that
	\[
	L_\Phi\vcentcolon = L_\Phi^1=L_\Phi^2.
	\]
	Let $2\le k\le m+1$, and suppose that
	\[
	h_1,\ldots,h_k\in C^{m+2,\alpha}(\partial\Omega).
	\]
	For $j=1,2$ and every nonempty subset
	$J\subset \{1,\ldots,k\}$, define
	\[
	U_J^j
	:=
	d^{|J|}\mathcal S_j(\phi)[h_\ell]_{\ell\in J}.
	\]
	Then $U_J^j|_{\partial\Omega}=0$ for $|J|\ge2$, and
	\[
	L_\Phi U_J^j=F_J^j
	\quad\text{in }\Omega,
	\]
	where
	\[
	\begin{split}
		F_J^j
		&=
		-\sum_{\substack{\pi\in\mathcal P(J):\\ |\pi|\ge2}}
		d^{|\pi|}(\det)(D^2\Phi)
		\big[
		D^2U_M^j
		\big]_{M\in\pi}
		\\
		&\quad
		+
		\sum_{\substack{\pi\in\mathcal P(J):\\ |\pi|\ge2}}
		d_{(z,p)}^{|\pi|}a_j(x,\Phi,\nabla\Phi)
		\big[
		(U_M^j,\nabla U_M^j)
		\big]_{M\in\pi}.
	\end{split}
	\]
	Assume that the $k$-th variations of the nonlinear Cauchy data sets at the
	background $\Phi$ agree. Equivalently,
	\begin{equation}
	\label{eq: equal Cauchy data}
	\partial_\nu U_{{1,\ldots,k}}^1
	=
	\partial_\nu U_{{1,\ldots,k}}^2
	\quad\text{on }\partial\Omega.
	\end{equation}
	Then, for every solution $w\in H^1(\Omega)$ of the adjoint problem
	\[
	L_\Phi^*w=0
	\quad\text{in }\Omega,
	\]
	there holds
	\[
	\int_\Omega
	\bigl(
	F_{{1,\ldots,k}}^1
	-
	F_{{1,\ldots,k}}^2
	\bigr)w\,dx
	=
	0.
	\]
\end{proposition}

\begin{proof}
	Set $J=\{1,\ldots,k\}$. By the higher order linearization formula,
	\[
	L_\Phi U_J^j=F_J^j,
	\qquad
	U_J^j|_{\partial\Omega}=0,
	\]
	for $j=1,2$; see Theorem~\ref{thm: local linearization MA}. Since the $k$-th variations of the nonlinear Cauchy data sets
	agree, we have
	\[
	\partial_\nu U_J^1
	=
	\partial_\nu U_J^2
	\quad\text{on }\partial\Omega.
	\]
	As the leading coefficient 
	$\mathcal C_\Phi$ is common, this also gives
	\[
	\nu\cdot\mathcal C_\Phi\nabla U_J^1
	=
	\nu\cdot\mathcal C_\Phi\nabla U_J^2
	\quad\text{on }\partial\Omega.
	\]
	Let $w\in H^1(\Omega)$ solve $L_\Phi^*w=0$. The divergence theorem gives
	\[
	\int_\Omega F_J^j w\,dx
	=
	 \int_\Omega (L_\Phi U_J^j)w\,dx=
	\int_{\partial\Omega}
	\nu\cdot\mathcal C_\Phi\nabla U_J^j\, w\,d\mathcal{H}^{n-1}.
	\]
	Subtracting the identities for $j=1,2$ and using equality of the conormal
	derivatives (see~equations \eqref{eq: equal Cauchy data} and \eqref{eq: decomposition of gradient}) yields
	\[
	\int_\Omega
	(F_J^1-F_J^2)w\,dx
	=
	0.
	\]
\end{proof}

\begin{corollary}[Isolation of the $k$-th derivatives of $a$]
	\label{cor: kth derivative identity}
	Assume the hypotheses of Proposition~\ref{prop: higher order Alessandrini}.
	Suppose moreover that for all $1\le r\le k-1$,
	\begin{equation}
	\label{eq: assump on derivatives of aj}
	d_{(z,p)}^r a_1(x,\Phi,\nabla\Phi)
	=
	d_{(z,p)}^r a_2(x,\Phi,\nabla\Phi)
	\quad\text{in }\Omega.
	\end{equation}
	Then
	\[
	\int_\Omega
	\Bigl(
	d_{(z,p)}^k a_1(x,\Phi,\nabla\Phi)
	-
	d_{(z,p)}^k a_2(x,\Phi,\nabla\Phi)
	\Bigr)
	\big[
	(v_1,\nabla v_1),\ldots,(v_k,\nabla v_k)
	\big]
	\,w\,dx
	=
	0,
	\]
	for all $H^1(\Omega)$-solutions of the adjoint problem $L_\Phi^*w=0$ in $\Omega$. Here, we use the shorthand notation
	\[
		v_\ell\vcentcolon = v_{h_\ell}=d\mathcal S(\phi)h_\ell,\quad h_\ell\in C^{m+2,\alpha}(\partial\Omega),
	\]
	which solves
	\[
	\begin{cases}
		L_\Phi v_\ell=0,
		&\quad\text{on }\Omega,\\
		v_\ell=h_\ell
		&\quad\text{on }\partial\Omega.
	\end{cases}
	\]
\end{corollary}

\begin{proof}
	Let $J=\{1,\ldots,k\}$. By Proposition~\ref{prop: higher order Alessandrini}, we have
	\[
	\int_\Omega
	\bigl(F_J^1-F_J^2\bigr)w\,dx
	=
	0
	\]
	for every solution 
	\[
	L_\Phi^* w=0
	\quad\text{in }\Omega.
	\]
	We now analyze the difference $F_J^1-F_J^2$. By the higher order
	linearization formula,
	\[
	\begin{split}
		F_J^j
		&=
		-\sum_{\substack{\pi\in\mathcal P(J):\\ |\pi|\ge2}}
		d^{|\pi|}(\det)(D^2\Phi)
		\bigl[D^2U_M^j\bigr]_{M\in\pi}
		\\
		&\quad+
		\sum_{\substack{\pi\in\mathcal P(J):\\ |\pi|\ge2}}
		d^{|\pi|}_{(z,p)}a_j(x,\Phi,\nabla\Phi)
		\bigl[
		(U_M^j,\nabla U_M^j)
		\bigr]_{M\in\pi}.
	\end{split}
	\]
	We first claim that
	\[
	U_M^1=U_M^2
	\qquad\text{for every }M\subset J\text{ with }1\le |M|\le k-1.
	\]
	For $|M|=1$, this follows from the equality
	$L_\Phi^1=L_\Phi^2=L_\Phi$ and the uniqueness of the Dirichlet problem.
	Assume the claim has been proved for all nonempty $M'\subset J$ with
	$|M'|\le r-1$. If $|M|=r$, then the equations
	for $U_M^1$ and $U_M^2$ have the same right-hand side (see~\eqref{eq: PDE for UJ}). In fact, \eqref{eq: assump on derivatives of aj} ensures that all
	derivatives $d^\ell_{(z,p)}a_1$ and $d^\ell_{(z,p)}a_2$ agree along $(\Phi,\nabla \Phi)$ for
	$\ell\le r\le k-1$, and all lower variations agree by the induction
	hypothesis. Moreover, both $U_M^1$ and $U_M^2$ have zero boundary value.
	As observed in the proof of Theorem~\ref{thm: local linearization MA},
	\[
	(L_\Phi,\cdot|_{\partial\Omega})\colon C^{m+2,\alpha}(\overline{\Omega})\to C^{m,\alpha}(\overline{\Omega})\times C^{m+2,\alpha}(\partial\Omega)
	\]
	is an isomorphism and thus $U_M^1=U_M^2$. Consequently, every term in $F_J^1-F_J^2$ corresponding to a partition
	$\pi$ with $|\pi|\leq k-1$ cancels.

	It remains to consider the only partition of $J$ with $|\pi|=k$, namely
	\[
	\pi=\bigl\{\{1\},\ldots,\{k\}\bigr\}.
	\]
	For this partition,
	\[
	U_{\{\ell\}}^j
	=
	v_\ell,
	\qquad
	\ell=1,\ldots,k,
	\]
	where $v_\ell$ is the solution of
	\[
	L_\Phi v_\ell=0,
	\qquad
	v_\ell=h_\ell
	\quad\text{on }\partial\Omega.
	\]
	Since $L_\Phi=L^1_\Phi=L^2_\Phi $, the functions $v_\ell$ are independent of $j$. 

	The determinant contribution for this partition also cancels, because it depends
	only on the common background $\Phi$ and on the common first variations
	$v_1,\ldots,v_k$. Hence the only remaining term in
	$F_J^1-F_J^2$ is
	\[
	\Bigl(
	d_{(z,p)}^k a_1(x,\Phi,\nabla\Phi)
	-
	d_{(z,p)}^k a_2(x,\Phi,\nabla\Phi)
	\Bigr)
	\big[
	(v_1,\nabla v_1),\ldots,(v_k,\nabla v_k)
	\big].
	\]
	Thus
	\[
	F_J^1-F_J^2
	=
	\Bigl(
	d_{(z,p)}^k a_1(x,\Phi,\nabla\Phi)
	-
	d_{(z,p)}^k a_2(x,\Phi,\nabla\Phi)
	\Bigr)
	\big[
	(v_1,\nabla v_1),\ldots,(v_k,\nabla v_k)
	\big].
	\]
	Substituting this identity into
	\[
	\int_\Omega
	(F_J^1-F_J^2)w\,dx
	=
	0
	\]
	gives
	\[
	\int_\Omega
	\Bigl(
	d_{(z,p)}^k a_1(x,\Phi,\nabla\Phi)
	-
	d_{(z,p)}^k a_2(x,\Phi,\nabla\Phi)
	\Bigr)
	\big[
	(v_1,\nabla v_1),\ldots,(v_k,\nabla v_k)
	\big]
	w\,dx
	=
	0.
	\]
	This proves the claim.
\end{proof}

\section{Uniqueness results for inverse problems}
\label{sec: Uniqueness}

\subsection{Anisotropic and lower-order uniqueness consequences}

We first record the natural boundary-fixing invariances of the anisotropic
linearized problem. Recall that we write $\sym$ for the class of symmetric $n\times n$ matrices.

\begin{proposition}[Boundary-fixing gauge invariance]
	\label{prop: boundary fixing gauge}
	Let
	\[
        \mathcal C\in L^\infty(\Omega;\sym),
	\qquad
	b\in L^\infty(\Omega;\mathbb R^n),
	\qquad
	q\in L^\infty(\Omega),
	\]
	where $\mathcal C$ is uniformly elliptic, and consider the second order differential operator
	\[
	L_{\mathcal C,b,q}u
	=
	\operatorname{div}(\mathcal C\nabla u)-b\cdot\nabla u-qu,\quad u\in H^1(\Omega).
	\]
	Assume that the associated Dirichlet problem is well posed in $H^1$, and denote its conormal
	Dirichlet-to-Neumann map by $\Lambda_{\mathcal C,b,q}$.

	\begin{enumerate}[(i)]
	\item \label{item: diffeomorphism gauge} Let
	$F\colon\Omega\to\Omega$ be an orientation-preserving boundary-fixing
	diffeomorphism of either of the following types:
	\begin{enumerate}[(a)]
		\item\label{case: C2 diffeo} $F$ extends to a $C^2$-diffeomorphism of $\overline\Omega$ and
		$F|_{\partial\Omega}=\mathrm{Id}$;
		\item\label{case: H1 diffeo} $n=2$ and $F$ is a boundary-fixing $H^1$-diffeomorphism in
		the sense of \cite{astala2005anisotropic}.
	\end{enumerate}
	In both cases, we define 
    \[
        J_F(x)\vcentcolon =\det DF(x)>0
    \]
    for a.e.~$x\in\Omega$ and let
	\[
	\begin{split}
	(F_*\mathcal C)(F(x))
	&\vcentcolon =\frac{DF(x)\mathcal C(x)DF(x)^T}{J_F(x)},\\
	(F_*b)(F(x))
	&\vcentcolon =\frac{DF(x)b(x)}{J_F(x)},\\
	(F_*q)(F(x))
	&\vcentcolon =\frac{q(x)}{J_F(x)}.
	\end{split}
    \]
	Moreover, we set
	\[
	F_*(\mathcal C,b,q) \vcentcolon =(F_*\mathcal C,F_*b,F_*q).
	\]
	Assume that the pushforward coefficients $F_*(\mathcal C,b,q)$ belong to the same coefficient
	class as $(\mathcal C,b,q)$ and that the related Dirichlet problem is well posed in $H^1$. Then
	\[
	\Lambda_{F_*(\mathcal C,b,q)}=\Lambda_{\mathcal C,b,q}.
	\]

	\item \label{item: conjugation gauge} Assume additionally that
	$\mathcal C\in C^1(\overline\Omega;\sym)$.
	Let
	$\psi\in C^2(\overline\Omega)$ satisfy
	\[
	\psi|_{\partial\Omega}=0,
	\qquad
	\nu\cdot\mathcal C\nabla\psi|_{\partial\Omega}=0,
	\]
	and define the conjugation gauge
	\[
	\mathcal G_\psi(\mathcal C,b,q)
	\vcentcolon =
	(\mathcal C,b_\psi,q_\psi),
	\]
	where
	\[
	b_\psi\vcentcolon=b-2\mathcal C\nabla\psi
	\]
	and
	\[
	q_\psi
	\vcentcolon=q+b\cdot\nabla\psi
	-\operatorname{div}(\mathcal C\nabla\psi)
	-\mathcal C\nabla\psi\cdot\nabla\psi.
	\]
	Then
	\[
	L_{\mathcal G_\psi(\mathcal C,b,q)}
	=e^{-\psi}L_{\mathcal C,b,q}e^\psi
	\qquad\text{and}\qquad
	\Lambda_{\mathcal G_\psi(\mathcal C,b,q)}
	=\Lambda_{\mathcal C,b,q}.
	\]
	\end{enumerate}

	Consequently, two coefficient triples related by
	\begin{equation}
	\label{eq: full coefficient gauge}
	(\mathcal C_2,b_2,q_2)
	=
	\mathcal G_\psi\bigl(F_*(\mathcal C_1,b_1,q_1)\bigr)
	\end{equation}
	for such $F$ and $\psi$ have the same Dirichlet-to-Neumann map, provided
	$F_*\mathcal C_1$ has the regularity required in
	\ref{item: conjugation gauge}; here the boundary condition on $\psi$ is
	understood with respect to $F_*\mathcal C_1$. We call such triples
	\emph{equivalent under the
	boundary-fixing diffeomorphism--conjugation gauge}.
\end{proposition}

\begin{proof}
	For \ref{item: diffeomorphism gauge}, set $y=F(x)$ and
	$v=u\circ F$. When the coefficients, $F$, and $u$ are smooth, the Piola
	transformation formula
	\cite[Section~8.1.4]{EvansPDE} gives the strong identity
	\[
	L_{F_*(\mathcal C,b,q)}u(F(x))
	=
	\frac{1}{J_F(x)}L_{\mathcal C,b,q}v(x).
	\]
	Hence the two equations correspond in the strong sense.

	The same conclusion holds for weak solutions. Indeed, if
	\[
	B_{\mathcal C,b,q}(w,\varphi)
	\vcentcolon=
	\int_\Omega
	\bigl(
	\mathcal C\nabla w\cdot\nabla\varphi
	+b\cdot\nabla w\,\varphi
		+qw\varphi
		\bigr)\,dx,
	\]
	then we claim that
	\[
	B_{F_*(\mathcal C,b,q)}(u,\eta)
	=
	B_{\mathcal C,b,q}(u\circ F,\eta\circ F)
	\]
	for $u,\eta\in H^1(\Omega)$. In the $C^2$ case, this is the usual change of
	variables, followed by density.

	Consider now the planar $H^1$ case. Since both $\mathcal C$ and
	$F_*\mathcal C$ are bounded and uniformly elliptic, the pushforward relation
	implies
	\[
	\lVert DF(x)\rVert^2\leq KJ_F(x)
	\qquad\text{for almost every }x\in\Omega
	\]
	with a constant $K<\infty$. Hence $F$ is quasiconformal; this is precisely
	the observation in \cite[eqs.~(8)--(9)]{astala2005anisotropic}. In
	particular, $F$ and $F^{-1}$ satisfy the area formula \cite{evans1991measure}, and the corresponding
	composition operators act on $H^1(\Omega)$ and $H_0^1(\Omega)$. The Sobolev
	chain rule gives
	\[
	\begin{split}
	\nabla(u\circ F)
	&=DF^T(\nabla u)\circ F,\\
	\nabla(\eta\circ F)
	&=DF^T(\nabla\eta)\circ F
	\end{split}
	\qquad\text{almost everywhere}.
	\]
	Using these identities, the definitions of the pushforward coefficients,
	and the area formula, we obtain
	\[
	\begin{split}
	\int_\Omega
	(F_*\mathcal C)\nabla u\cdot\nabla\eta\,dy
	&=
	\int_\Omega
	\mathcal C\nabla(u\circ F)\cdot\nabla(\eta\circ F)\,dx,\\
	\int_\Omega
	(F_*b)\cdot\nabla u\,\eta\,dy
	&=
	\int_\Omega
	b\cdot\nabla(u\circ F)\,(\eta\circ F)\,dx,\\
	\int_\Omega
	(F_*q)u\eta\,dy
	&=
	\int_\Omega
	q(u\circ F)(\eta\circ F)\,dx.
	\end{split}
	\]
	This proves the claimed bilinear identity and hence the correspondence of
	the equations in the weak sense. Since $F$ fixes $\partial\Omega$,
	composition with $F$ preserves the boundary traces.
	Applying the preceding bilinear identity
	to extensions of arbitrary boundary test functions shows that
	\[
	\bigl\langle\Lambda_{F_*(\mathcal C,b,q)}f,h\bigr\rangle
	=
	\bigl\langle\Lambda_{\mathcal C,b,q}f,h\bigr\rangle
	\]
	for all $f,h\in H^{1/2}(\partial\Omega)$. This proves the asserted equality
	of the Dirichlet-to-Neumann maps as maps into
	$H^{-1/2}(\partial\Omega)$.

	For \ref{item: conjugation gauge}, the product rule gives
	\[
	e^{-\psi}L_{\mathcal C,b,q}(e^\psi u)
	=
	\operatorname{div}(\mathcal C\nabla u)
	-b_\psi\cdot\nabla u-q_\psi u,
	\]
	which proves the operator identity in the strong sense for smooth $u$.
	The same identity holds in the sense of distributions for $u\in H^1(\Omega)$,
	as follows by testing against compactly supported smooth functions and using
	density. Thus, if $u$ is a weak solution of
	$L_{\mathcal G_\psi(\mathcal C,b,q)}u=0$, then
	$v=e^\psi u$ is a weak solution of $L_{\mathcal C,b,q}v=0$. Moreover, the
	boundary traces satisfy $v=u$ and, for smooth solutions, the conormal traces satisfy pointwise
	\[
	\nu\cdot\mathcal C\nabla v
	=
	e^\psi\bigl(\nu\cdot\mathcal C\nabla u
	+u\,\nu\cdot\mathcal C\nabla\psi\bigr)
	=
	\nu\cdot\mathcal C\nabla u.
	\]
	For weak solutions, the same conormal-trace identity holds in
	$H^{-1/2}(\partial\Omega)$ by Green's identity and density.
	Thus the Dirichlet and conormal Neumann traces agree, proving the second
	equality of Dirichlet-to-Neumann maps. The final assertion follows by applying
	\ref{item: diffeomorphism gauge} and then
	\ref{item: conjugation gauge}.
\end{proof}

\begin{theorem}[Anisotropic uniqueness and reduction of the gauge]
	\label{thm: anisotropic uniqueness gauge}
	Assume the hypotheses of Theorem~\ref{thm: reduction-principle-full}.
	Suppose, in addition, that the associated anisotropic linear inverse problem satisfies the
	following uniqueness property: if
	\[
	\Lambda_{\Phi_1}^1=\Lambda_{\Phi_2}^2,
	\]
	then the two linearized coefficient triples  \[
	(\mathcal C_{\Phi_1},b_{\Phi_1}^1,q_{\Phi_1}^1)
	\quad\text{and}\quad
	(\mathcal C_{\Phi_2},b_{\Phi_2}^2,q_{\Phi_2}^2)
	\]
    are gauge equivalent in the sense
	of Proposition~\ref{prop: boundary fixing gauge}, that is, they are related as in
	\eqref{eq: full coefficient gauge}.

	If the local Cauchy data satisfy
	\[
	\mathscr C_{a_1}^{(\phi,\partial_\nu\Phi_1|_{\partial\Omega})}
	=
	\mathscr C_{a_2}^{(\phi,\partial_\nu\Phi_2|_{\partial\Omega})},
	\]
	then  
    \[
	(\mathcal C_{\Phi_1},b_{\Phi_1}^1,q_{\Phi_1}^1)
	\quad\text{and}\quad
	(\mathcal C_{\Phi_2},b_{\Phi_2}^2,q_{\Phi_2}^2)
	\]
    are gauge equivalent with trivial diffeomorphism component, $F=\mathrm{Id}$. More specifically, 
	\begin{equation}
    \label{eq: aniso conjugation relation}
	(\mathcal C_{\Phi_2},b_{\Phi_2}^2,q_{\Phi_2}^2)
	=
	\mathcal G_\psi
	(\mathcal C_{\Phi_1},b_{\Phi_1}^1,q_{\Phi_1}^1).
	\end{equation}
    
    In the pure conductivity case
	\[
	b_{\Phi_j}^j=q_{\Phi_j}^j=0,
	\qquad j=1,2,
	\]
	the conjugate function $\psi$ is also trivial, $\psi=0$, and one has
    \[
        (\mathcal C_{\Phi_2},b_{\Phi_2}^2,q_{\Phi_2}^2)
	=
	(\mathcal C_{\Phi_1},b_{\Phi_1}^1,q_{\Phi_1}^1).
    \]
	Additionally, in this case, the conditional linear uniqueness hypothesis
	above need not be imposed separately in either of the following settings:
	\begin{enumerate}[(i)]
		\item\label{two-dim case} $n=2$, under the hypotheses of the planar anisotropic conductivity
		theorem \cite{astala2005anisotropic};
		\item\label{higher-dim case} $n\geq 3$, under the real-analytic hypotheses of
		 \cite{lee1989determining};
	\end{enumerate}
	see Remark~\ref{rem: positive uniqueness results}. In both settings, equality of the nonlinear Cauchy data therefore implies exact
	equality of the complete linearized coefficient triples.
 \end{theorem}

 \begin{remark}[Existence and compatibility of background solutions]
    \label{rem: existence backgrounds anisotropic theorem}
    The existence assumption on the background solutions in
    Theorem~\ref{thm: anisotropic uniqueness gauge} is satisfied, for instance,
    under the stronger hypotheses of
    Remark~\ref{rem: existence background}. More precisely, if $\Omega$ is
    uniformly convex and of class $C^{4,\alpha}$,
    $\phi\in C^{4,\alpha}(\partial\Omega)$, and
    \[
        a_j\in
        C^{3,\alpha}(\overline\Omega\times\mathbb R\times\mathbb R^n),
        \qquad
        a_j\geq\lambda>0,
        \qquad
        \partial_z a_j\geq0,
        \qquad j=1,2,
    \]
    then there exists a unique uniformly convex solution
    $\Phi_j\in C^{4,\alpha}(\overline\Omega)$ of
    \[
        \det D^2\Phi_j=a_j(x,\Phi_j,\nabla\Phi_j)
        \quad\text{in }\Omega,
        \qquad
        \Phi_j=\phi
        \quad\text{on }\partial\Omega.
    \]
    These backgrounds satisfy the regularity requirements of
Theorem~\ref{thm: anisotropic uniqueness gauge} for $m=0,1$.
Under corresponding higher-regularity assumptions on the domain, the boundary
value, and the nonlinearities, elliptic regularity gives backgrounds in
$C^{m+2,\alpha}(\overline\Omega)$ for larger values of $m$.

    This existence result produces the individual background solutions, but it
    does not by itself imply the compatibility conditions
    \eqref{eq: compatibility condition a} and
    \eqref{eq: compatibility condition b}. These remain assumptions relating
    the two backgrounds and nonlinearities in
    Theorem~\ref{thm: reduction-principle-full}, and hence in
    Theorem~\ref{thm: anisotropic uniqueness gauge}. The same distinction between the existence of individual backgrounds and the
compatibility of a pair of backgrounds applies to the other inverse problems
considered in the remainder of Section~\ref{sec: Uniqueness}.
\end{remark}

\begin{proof}[Proof of Theorem~\ref{thm: anisotropic uniqueness gauge}]
	By Theorem~\ref{thm: reduction-principle-full}, equality of the nonlinear
	Cauchy data sets implies
	\[
	\Lambda_{\Phi_1}^1=\Lambda_{\Phi_2}^2.
	\]
	The assumed linear uniqueness statement therefore gives a boundary-fixing
	diffeomorphism $F$ and a function $\psi$ for which
	\eqref{eq: full coefficient gauge} holds. Since conjugation does not change the
	principal coefficient, this relation implies
	\begin{equation}
	\label{eq: principal gauge relation proof}
	\mathcal C_{\Phi_2}=F_*\mathcal C_{\Phi_1}.
	\end{equation}
    Both principal coefficients are divergence free by the Piola identity. For each $k\in\{1,\ldots,n\}$, the
	coordinate function $y\mapsto y^k$ is a weak
	$\mathcal C_{\Phi_2}$-harmonic function, since
	\[
	\Div_{y}(\mathcal C_{\Phi_2}e_k)=0
	\]
	in the sense of distributions; see Remark~\ref{rem: divergence form}. Applying the weak change-of-variables identity
	from the proof of Proposition~\ref{prop: boundary fixing gauge} to
	$u(y)=y^k$ shows that
	\[
	F^k=y^k\circ F
	\]
	is weakly $\mathcal C_{\Phi_1}$-harmonic. This conclusion is valid both for the
	smooth diffeomorphisms in Proposition~\ref{prop: boundary fixing gauge}
	\ref{item: diffeomorphism gauge}, \ref{case: C2 diffeo} and for the planar $H^1$
	diffeomorphisms in \ref{item: diffeomorphism gauge}, \ref{case: H1 diffeo}. The coordinate function
	$x\mapsto x^k$ is likewise weakly $\mathcal C_{\Phi_1}$-harmonic because
	\[
	\operatorname{div}_{x}(\mathcal C_{\Phi_1}e_k)=0
	\]
	in the sense of distributions. Since $F$ is boundary fixing, $F^k$ and $x^k$
	have the same $H^{1/2}(\partial\Omega)$ trace. Uniqueness for the uniformly
	elliptic Dirichlet problem gives $F^k=x^k$ in $\Omega$ for every $k$, and
	therefore $F=\mathrm{Id}$.
	Equation~\eqref{eq: principal gauge relation proof} now yields
	$\mathcal C_{\Phi_1}=\mathcal C_{\Phi_2}$, while
	\eqref{eq: full coefficient gauge} reduces to the asserted conjugation relation \eqref{eq: aniso conjugation relation}.

	Finally, if $b_{\Phi_j}^j=q_{\Phi_j}^j=0$ for $j=1,2$, comparison of the
	first-order coefficients in the conjugation relation gives
	\[
	0=-2\mathcal C_{\Phi_1}\nabla\psi.
	\]
	Uniform ellipticity implies $\nabla\psi=0$, and the boundary condition
	$\psi|_{\partial\Omega}=0$ gives $\psi=0$. Thus the complete triples agree.

	In the two-dimensional case, \cite{astala2005anisotropic} supplies the required
	boundary-fixing diffeomorphism relation for the conductivity tensors. Under the
	real-analytic hypotheses in dimensions $n\geq3$, the corresponding relation
	follows from \cite{lee1989determining}. In either case the preceding
	divergence-free argument forces the diffeomorphism to be the identity, proving
	the final assertion.
\end{proof}

\begin{remark}[Gauge-reduction argument]
	\label{rem: gauge rigidity}
	The proof of Theorem~\ref{thm: anisotropic uniqueness gauge} separates two distinct parts. First, the assumed
	linear inverse theorem supplies the gauge relation
	\eqref{eq: full coefficient gauge}. Second, conjugation leaves the principal coefficient unchanged, so the
	gauge relation gives
	\[
	\mathcal C_{\Phi_2}=F_*\mathcal C_{\Phi_1}.
	\]
	Only at this stage is the divergence-free coordinate-function argument applied,
	and it forces $F=\mathrm{Id}$. Thus the nonlinear data determine the principal
	cofactor matrix exactly whenever the underlying linear theorem gives uniqueness
	up to the boundary-fixing gauge of Proposition~\ref{prop: boundary fixing gauge}.

    Additionally, in the pure conductivity case, comparison of
	the transformed first-order coefficients forces $\nabla\psi=0$, so the boundary
	normalization gives $\psi=0$ and the full triples agree.
\end{remark}

\begin{remark}[{Assumptions of anisotropic uniqueness results of Theorem~\ref{thm: anisotropic uniqueness gauge}}]
	\label{rem: positive uniqueness results}
	In dimension $n=2$, let $\Omega\subset\mathbb R^2$ be a bounded simply
	connected domain, and let
	\[
    \mathcal C\in L^\infty(\Omega;\mathbb S^2)
	\]
	be uniformly elliptic. Then the Dirichlet-to-Neumann map determines
	$\mathcal C$ up to pushforward by a boundary-fixing
	$H^1$-diffeomorphism \cite{astala2005anisotropic}. This is the planar
	anisotropic uniqueness result used in the proof of
	Theorem~\ref{thm: anisotropic uniqueness gauge}, and it encodes the relevant assumptions we use for $n=2$.

	In dimensions $n\ge3$, the general anisotropic Calderón problem remains open
	for arbitrary anisotropic $L^{\infty}$ conductivities. Positive results under
	additional assumptions on the principal coefficient include:
	\begin{enumerate}[(i)]
		\item real-analytic anisotropic conductivities or metrics \cite{lee1989determining};
		\item structured anisotropies of the form $A(x,\theta(x))$, including piecewise analytic or piecewise constant classes \cite{alessandrini2001determining,alessandrini2017uniqueness};
		\item conformally transversally anisotropic or admissible geometries
		\cite{dos2009limiting};
		\item boundary determination and stability under additional a priori
		structure \cite{alessandrini2009local}.
	\end{enumerate}
	For the higher-dimensional conclusion in
Theorem~\ref{thm: anisotropic uniqueness gauge}, we use the real-analytic
result \cite{lee1989determining}. In that setting, we suppose that the boundary $\partial\Omega$ and the
boundary value $\phi$ are real analytic and that $a_j$ is real analytic in a
neighborhood of the closed background jet
\[
\{(x,\Phi_j(x),\nabla\Phi_j(x)):x\in\overline\Omega\}.
\]
The nonlinear operator 
\[
    F(x,z,p,r)=\det r-a(x,z,p)
\]
is then real analytic, and its linearization with respect to the Hessian variable along $\Phi_j$ is 
\[
    \frac{\partial F}{\partial r}(x,\Phi_j,\nabla\Phi_j,D^2\Phi_j)=\cof D^2\Phi_j.
\]
Since $\Phi_j$ is uniformly convex, this matrix is uniformly positive definite. Thus the equation is uniformly elliptic along $\Phi_j$, while the Dirichlet boundary condition satisfies the complementing condition. Analytic regularity for nonlinear elliptic boundary value problems \cite{morrey1958analyticityI,morrey1958analyticityII,wang2006analyticity} therefore implies that $\Phi_j$ is real analytic up to the boundary. Consequently,
\[
    \mathcal C_{\Phi_j}=\cof D^2\Phi_j
\]
is real analytic, as required.

The same gauge-reduction argument may be
	considered in the other settings above whenever the corresponding
	linear result provides the required global gauge relation for the
	cofactor-Hessian conductivities.
 \end{remark}

The next result concerns the determination of lower-order coefficients with fixed background solution $\Phi$.

\begin{theorem}[Fixed-background lower-order uniqueness]
	\label{thm: fixed background lower order uniqueness}
	Assume the hypotheses of Theorem~\ref{thm: reduction-principle-full} in the
	special case
	\[
	\Phi_1=\Phi_2=\vcentcolon \Phi.
	\]
	Let
	\[
	L_\Phi^j v
	\vcentcolon =
	\Div(\mathcal C_\Phi\nabla v)
	-
	b_\Phi^j\cdot\nabla v
	-
	q_\Phi^j v,
	\qquad j=1,2.
	\]
	Assume that the fixed-principal-part linear inverse problem is uniquely solvable
	for the lower-order coefficients, in the following sense: if
	\[
	\int_\Omega
	(b_\Phi^1-b_\Phi^2)\cdot\nabla v\,w\,dx
	+
	\int_\Omega
	(q_\Phi^1-q_\Phi^2)vw\,dx
	=
	0
	\]
	for all $H^1$ solutions
	\[
	L_\Phi^1v=0,
	\qquad
	(L_\Phi^2)^*w=0,
	\]
	then
	\[
	b_\Phi^1=b_\Phi^2,
	\qquad
	q_\Phi^1=q_\Phi^2
	\quad\text{in }\Omega.
	\]

	If
	\[
	\mathscr C_{a_1}^{(\phi,\partial_\nu\Phi|_{\partial\Omega})}
	=
	\mathscr C_{a_2}^{(\phi,\partial_\nu\Phi|_{\partial\Omega})},
	\]
	then
	\[
	\nabla_p a_1(x,\Phi,\nabla\Phi)
	=
	\nabla_p a_2(x,\Phi,\nabla\Phi),
	\]
	and
	\[
	\partial_z a_1(x,\Phi,\nabla\Phi)
	=
	\partial_z a_2(x,\Phi,\nabla\Phi)
	\]
	in $\Omega$.
\end{theorem}

\begin{proof}
	Since the background solution is common, we have
	\[
	\mathcal C_\Phi \vcentcolon = \mathcal C_{\Phi_1}=\mathcal C_{\Phi_2}.
	\]
	By Theorem~\ref{thm: reduction-principle-full}, equality of the nonlinear Cauchy
	data sets implies
	\[
	\int_\Omega
	(b_\Phi^1-b_\Phi^2)\cdot\nabla v\,w\,dx
	+
	\int_\Omega
	(q_\Phi^1-q_\Phi^2)vw\,dx
	=
	0
	\]
	for all weak solutions
	\[
	L_\Phi^1v=0,
	\qquad
	(L_\Phi^2)^*w=0.
	\]
	The assumed fixed-principal-part uniqueness property gives
	\[
	b_\Phi^1=b_\Phi^2,
	\qquad
	q_\Phi^1=q_\Phi^2
	\quad\text{in }\Omega.
	\]
	By definition,
	\[
	b_\Phi^j=\nabla_p a_j(x,\Phi,\nabla\Phi),
	\qquad
	q_\Phi^j=\partial_z a_j(x,\Phi,\nabla\Phi),
	\]
	which gives the desired conclusion.
\end{proof}

The goal of the remaining sections is to present concrete applications of the above abstract uniqueness theorems. We shall focus on the setup of Theorem~\ref{thm: fixed background lower order uniqueness}.

\subsection{First-order recovery of semilinearities along quadratic backgrounds}

\begin{theorem}[First-order recovery along quadratic background]
	\label{thm: quadratic-background-zero-order}
	Let $n\ge 3$,  $0<\alpha<1$, and $a_1,a_2\in C^{1,\alpha}(\overline{\Omega}\times\R\times\R^n)$
	be positive nonlinearities. 

	Assume that $a_j$ is independent of $p$, that is
	\[
	a_j=a_j(x,z),\qquad j=1,2,
	\]
	and that 
	\[
		\Phi(x)=\frac{ |x|^2}{2}
	\]
	is a common background solution, implying
	\[
	a_j\left(x,\frac{ |x|^2}{2}\right)=1
	\quad\text{in }\Omega,
	\qquad j=1,2.
	\]
	Suppose that
	\[
	q_j(x):=\partial_z a_j\left(x,\frac{ |x|^2}{2}\right)\geq 0.
	\]

	If the nonlinear Cauchy data sets agree near the background,
	\[
	\mathscr C_{a_1}^{(\phi,\partial_\nu\Phi|_{\partial\Omega})}
	=
	\mathscr C_{a_2}^{(\phi,\partial_\nu\Phi|_{\partial\Omega})},
	\]
	then
	\[
	\partial_z a_1\left(x,\frac{ |x|^2}{2}\right)
	=
	\partial_z a_2\left(x,\frac{ |x|^2}{2}\right)
	\quad\text{in }\Omega.
	\]
\end{theorem}

\begin{remark}
	Note that the assumptions of Theorem~\ref{thm: quadratic-background-zero-order} ensure that we may apply Theorem~\ref{thm: local linearization MA} with $m=0$ and hence the nonlinear Cauchy data sets are well-defined. The same observation will be used in Theorem~\ref{thm: optimal-transport-equilibrium}.
\end{remark}

\begin{proof}
	Since
	\[
	D^2\Phi=I,
	\]
	we have
	\[
	\mathcal C_\Phi=\cof D^2\Phi=I.
	\]
	Moreover, because $a_j$ is independent of the gradient variable,
	\[
	b_\Phi^j=\nabla_p a_j(x,\Phi,\nabla\Phi)=0.
	\]
	Thus the linearized operators are
	\[
	L_\Phi^j v
	=
	\Delta v-q_jv.
	\]

	By the reduction principle, equality of the nonlinear Cauchy data sets implies
	\[
	\int_\Omega (q_1-q_2)vw\,dx=0
	\]
	for all solutions
	\[
	(\Delta-q_1)v=0,
	\qquad
	(\Delta-q_2)w=0
	\quad\text{in }\Omega.
	\]
	By the classical Calderón uniqueness theorem for the Schrödinger equation in
	dimension $n\ge3$, this integral identity implies
	\[
	q_1=q_2
	\quad\text{in }\Omega;
	\]
	see \cite{sylvester1987global}. So the assertion follows from the definition of $q_j$.
\end{proof}

\begin{theorem}[Equilibrium optimal transport]
	\label{thm: optimal-transport-equilibrium}
	Suppose that $n\ge 3$,  $0<\alpha<1$, $\Omega\subset\R^n$ is a $C^{2,\alpha}$-domain, and $a_1,a_2\in C^{1,\alpha}(\overline{\Omega}\times\R^n)$ are positive nonlinearities. 
	Assume that
	\begin{equation}
	a_j(x,p)=\frac{\rho_j(x)}{\rho_j(p)},
	\qquad j=1,2,
	\end{equation}
	where $\rho_j\in C_b^{\infty}(\R^n)$ are positive and satisfy
	\[
	\rho_1|_{\partial\Omega}=\rho_2|_{\partial\Omega}.
	\]
	If the nonlinear Cauchy data sets agree near the background $\Phi(x)=|x|^2/2$,
	\[
	\mathscr C_{a_1}^{(\phi,\partial_\nu\Phi|_{\partial\Omega})}
	=
	\mathscr C_{a_2}^{(\phi,\partial_\nu\Phi|_{\partial\Omega})},
	\]
	then
	\[
	\rho_1=\rho_2
	\quad\text{in }\Omega.
	\]
\end{theorem}

\begin{proof}
	Notice that the nonlinear Cauchy data sets are well-defined, because
	\[
	\Phi(x)\vcentcolon = \frac{|x|^2}{2}
	\]
	is a common background solution of the nonlinear Monge–Ampère equation. This follows from $D^2\Phi=I$ and
	\[
	a_j(x,\nabla\Phi(x))
	=
	\frac{\rho_j(x)}{\rho_j(x)}
	=
	1.
	\]
	Moreover, we have
	\[
	q_\Phi^j
	=
	\partial_z a_j(x,\Phi,\nabla\Phi)
	=
	0,
	\]
	and
	\[
	b_\Phi^j
	=
	\nabla_p a_j(x,\Phi,\nabla\Phi)
	=
	-\nabla \log \rho_j(x).
	\]
	Hence the linearized equation is
	\[
	L_\Phi^j v
	=
	\Delta v+\nabla\log\rho_j\cdot\nabla v.
	\]
	Equivalently,
	\[
	L_\Phi^j v
	=
	\rho_j^{-1}\Div(\rho_j\nabla v).
	\]

	By the reduction principle, equality of the nonlinear Cauchy data sets implies
	equality of the linearized Cauchy data sets, and hence equality of the usual
	normal derivative data for
	\[
	\Div(\rho_j\nabla v)=0.
	\]
	Since
	\[
	\rho_1|_{\partial\Omega}=\rho_2|_{\partial\Omega},
	\]
	the equality of normal derivative data implies equality of the conductivity DN
	maps
	\[
	\Lambda_{\rho_1}^{\mathrm{cond}}
	=
	\Lambda_{\rho_2}^{\mathrm{cond}}.
	\]
	Recall that the conductivity DN map
	\[
		\Lambda_{\rho_j}^{\mathrm{cond}}\colon H^{1/2}(\partial\Omega)\to H^{-1/2}(\partial\Omega)
	\]
	is the extension of 
	\[
		\Lambda_{\rho_j}^{\mathrm{cond}}\varphi \vcentcolon= \rho_j\partial_\nu u_\varphi
	\]
	to boundary values in $H^{1/2}(\partial\Omega)$, where $u_\varphi$ solves
	\[
		\Div(\rho_j\nabla v)=0\quad\text{in }\Omega
	\]
	and has boundary value $u_\varphi|_{\partial\Omega}=\varphi$.

	By the Calderón uniqueness theorem for isotropic conductivities in dimension
	$n\ge3$, we conclude
	\[
	\rho_1=\rho_2
	\quad\text{in }\Omega;
	\]
	see \cite{kohn1984determining,sylvester1987global}.
\end{proof}

\begin{remark}[Beyond equilibrium]
	The preceding theorem treats the equilibrium transport configuration, where the
	source and target densities coincide and the optimal transport map is the identity.
	It would be interesting to extend this result to non-equilibrium transport states
	\[
	a(x,p)=\frac{\rho^+(x)}{\rho^-(p)},
	\]
	where the background solution $\Phi$ generates a nontrivial Brenier map
	\[
	T=\nabla\Phi.
	\]
	In such a setting, first-order linearization formally gives access to derivatives
	of the target density along the image $T(\Omega)$. Thus, several background
	transport states could in principle probe different regions of phase space.

	However, a general theorem of this type would require additional assumptions
	ensuring uniqueness for the corresponding anisotropic linearized Calderón
	problems, as well as suitable coverage conditions for the family of transport
	maps. We therefore restrict the present work to the equilibrium case, where the
	linearized equation is an isotropic conductivity equation and classical Calderón
	uniqueness applies directly.
\end{remark}

\subsection{Conditional higher-order uniqueness along a common background}

\begin{theorem}[Conditional recovery of the $k$-th derivatives]
	\label{thm: conditional kth derivative recovery}
	Suppose that the hypotheses of
	Proposition~\ref{prop: higher order Alessandrini} hold. Let
	$2\le k\le m+1$ and assume that the lower-order derivatives agree along the background jet, namely
	\[
	d_{(z,p)}^r a_1(x,\Phi,\nabla\Phi)
	=
	d_{(z,p)}^r a_2(x,\Phi,\nabla\Phi)
	\quad\text{in }\Omega,
	\qquad
	1\le r\le k-1.
	\]
	Assume moreover that the following product-density implication holds: if a
	symmetric $k$-linear form
	\[
	A(x)\in \operatorname{Sym}^k(\R\times \mathbb R^{n})^*
	\]
	satisfies
	\[
	\int_\Omega
	A(x)
	\big[
	(v_1,\nabla v_1),\ldots,(v_k,\nabla v_k)
	\big]
	w\,dx
	=
	0
	\]
	for all solutions
	\[
	L_\Phi v_\ell=0,\qquad \ell=1,\ldots,k,
	\]
	and
	\[
	L_\Phi^*w=0,
	\]
	then
	\[
	A=0
	\quad\text{in }\Omega.
	\]
	If the nonlinear Cauchy data sets have the same variations up to order $k$ at
	the background $\Phi$, then
	\[
	d_{(z,p)}^k a_1(x,\Phi,\nabla\Phi)
	=
	d_{(z,p)}^k a_2(x,\Phi,\nabla\Phi)
	\quad\text{in }\Omega.
	\]
\end{theorem}

\begin{proof}
	By Corollary~\ref{cor: kth derivative identity}, equality of the $k$-th
	variations of the nonlinear Cauchy data sets and equality of all lower-order
	derivatives imply
	\[
	\int_\Omega
	\Bigl(
	d_{(z,p)}^k a_1(x,\Phi,\nabla\Phi)
	-
	d_{(z,p)}^k a_2(x,\Phi,\nabla\Phi)
	\Bigr)
	\big[
	(v_1,\nabla v_1),\ldots,(v_k,\nabla v_k)
	\big]
	w\,dx
	=
	0
	\]
	for all solutions
	\[
	L_\Phi v_\ell=0,
	\qquad
	L_\Phi^*w=0.
	\]
	Applying the assumed product-density implication with
	\[
	A(x)
	= d_{(z,p)}^k a_1(x,\Phi,\nabla\Phi)
	-d_{(z,p)}^k a_2(x,\Phi,\nabla\Phi)
	\]
	gives
	\[
	A=0
	\quad\text{in }\Omega.
	\]
	Thus
	\[
	d_{(z,p)}^k a_1(x,\Phi,\nabla\Phi)
	=
	d_{(z,p)}^k a_2(x,\Phi,\nabla\Phi)
	\quad\text{in }\Omega.
	\]
\end{proof}

\subsection{Higher-order recovery of semilinearities along quadratic backgrounds}

Theorem~\ref{thm: conditional kth derivative recovery} reduces the recovery of
the $k$-th derivative
\[
d_{(z,p)}^k a(x,\Phi,\nabla\Phi)
\]
to a product-density property for solutions of the linearized equation and its
adjoint. In general, establishing such density statements appears to be a
difficult problem. However, in several special situations, classical complex
geometrical optics (CGO) constructions imply the required density and lead to
unconditional uniqueness results.

We illustrate this principle in the simplest setting.

\begin{theorem}[Higher-order recovery along quadratic background]
	\label{thm: higher-order-semilinear}
		Let $n\ge 3$, $m\in\N_0$,  $0<\alpha<1$, and $a_1,a_2\in C^{2m+1,\alpha}(\overline{\Omega}\times\R\times\R^n)$ be positive nonlinearities. Suppose additionally that $\Omega$ is connected. Assume that the nonlinearitites $a_j$ are independent of the $p$-variable, that is
	\[
	a_j=a_j(x,z),
	\qquad j=1,2,
	\]
	they satisfy the normalization condition
	\begin{equation}
	\label{eq: normalization cond}
	a_j\!\left(x,\frac{|x|^2}{2}\right)=1
	\quad\text{in }\Omega,
	\end{equation}
	and the nonnegativity condition
	\begin{equation}
	\label{eq: nonnegativity cond higher order}
		q_j\vcentcolon = \partial_z a_j\!\left(x,\frac{|x|^2}{2}\right)\geq 0\quad\text{in }\Omega.
	\end{equation}

	Assume that the nonlinear Cauchy data sets agree near the background,
	\[
	\mathscr C_{a_1}^{(\phi,\partial_\nu\Phi|_{\partial\Omega})}
	=
	\mathscr C_{a_2}^{(\phi,\partial_\nu\Phi|_{\partial\Omega})}.
	\]

	Then, for every integer $k$ satisfying
	\[
	1\le k\le m+1,
	\]
	there holds
	\[
	\partial_z^k a_1\!\left(x,\frac{|x|^2}{2}\right)
	=
	\partial_z^k a_2\!\left(x,\frac{|x|^2}{2}\right)
	\qquad\text{in }\Omega.
	\]
\end{theorem}

\begin{remark}[Real-analytic semilinearities]
	Suppose, in addition, that the nonlinearities $a_1,a_2$ are real-analytic in the variables $(x,z)$.
	Then Theorem~\ref{thm: higher-order-semilinear} implies that
	\[
	\partial_z^k a_1\!\left(x,\frac{|x|^2}{2}\right)
	=
	\partial_z^k a_2\!\left(x,\frac{|x|^2}{2}\right)
	\]
	for all $k\ge0$ and $x\in\Omega$.
	Hence $a_1$ and $a_2$ have identical Taylor expansions along the
	real-analytic hypersurface
	\[
	\Gamma
	=
	\left\{
	\left(x,\frac{|x|^2}{2}\right)
	\,;\,
	x\in\Omega
	\right\}.
	\]
	Since $\Gamma$ is a real-analytic hypersurface and all $z$-jets coincide on $\Gamma$, the analytic continuation theorem implies $a_1=a_2$.
\end{remark}

\begin{proof}
	First, we observe that \eqref{eq: normalization cond} guarantees that
	\[
	\Phi(x)\vcentcolon = \frac{|x|^2}{2}
	\]
	is a common background solution and hence the nonlinear Cauchy data sets are well-defined.

	Moreover, as noticed in Theorem~\ref{thm: optimal-transport-equilibrium}, the linearized operator $L_\Phi$ is given by
	\[
	L_\Phi^j
	=
	\Delta-q_j,
	\qquad
	q_j=
	\partial_z a_j\!\left(x,\frac{|x|^2}{2}\right),
	\]
	because $\mathcal C_\Phi=I$ and $a_j$ is independent of $p$.

	Moreover, by Theorem~\ref{thm: quadratic-background-zero-order},
	\[
	q\vcentcolon = q_1=q_2\quad\text{in }\Omega,
	\]
	and hence both nonlinearities have the same linearized operator
	\[
	L_\Phi
	=
	\Delta-q.
	\]
	This establishes the case $k=1$.

	We now proceed by induction on $k$. Let $m\in\N_0$ and $k\in \N$ satisfy $2\leq k\leq m+1$. Assume that
	\[
	\partial_z^r a_1\!\left(x,\frac{|x|^2}{2}\right)
	=
	\partial_z^r a_2\!\left(x,\frac{|x|^2}{2}\right)
	\]
	for all $1\le r\le k-1$. By
	Corollary~\ref{cor: kth derivative identity},
	\begin{equation}
	\label{eq: orthogonality relation}
	\int_\Omega
	\left[
	\partial_z^k a_1\!\left(x,\frac{|x|^2}{2}\right)
	-
	\partial_z^k a_2\!\left(x,\frac{|x|^2}{2}\right)
	\right]
	v_1\cdots v_k\,w\,dx
	=
	0
	\end{equation}
	for all solutions $v_\ell, w\in H^1(\Omega)$, $1\leq\ell\leq k$, of the Schr\"odinger equation 
	\begin{equation}
	\label{eq: schroedinger eq for multiple sols}
	(-\Delta+q)u=0\quad\text{in }\Omega.
	\end{equation}
	Since $q\geq 0$, the related Dirichlet problem is well-posed. Hence, let 
	\[
		v\vcentcolon = v_2=\ldots=v_k\in H^1(\Omega)
	\]
	be the unique solution of \eqref{eq: schroedinger eq for multiple sols} with nonzero boundary values $\psi\in H^{1/2}(\partial\Omega)\cap L^{\infty}(\partial\Omega)$. Note that the maximum principle \cite[Theorem~8.1]{GT} implies $v\in L^{\infty}(\Omega)$. We choose the remaining two solutions $v_1$ and $w$ to be the classical Sylvester--Uhlmann CGO solutions. Namely, for every
	$\rho\in\mathbb C^n$ satisfying
	\[
	\rho\cdot\rho=0,
	\qquad |\rho|\gg1,
	\]
	there exists a solution
	\[
	u_\rho(x)=e^{\rho\cdot x}(1+r_\rho(x))
	\]
	of
	\[
	(-\Delta+q)u_\rho=0
	\quad\text{in }\Omega,
	\]
	with
	\[
	\|r_\rho\|_{L^2(\Omega)}\to 0
	\quad\text{as }|\rho|\to\infty.
	\]
	Thus, \eqref{eq: orthogonality relation} yields
	\[
		\left[
		\partial_z^k a_1\!\left(x,\frac{|x|^2}{2}\right)
		-
		\partial_z^k a_2\!\left(x,\frac{|x|^2}{2}\right)
	\right] v^{k-1}=0\quad\text{in }\Omega.
	\]
	Suppose there exists $x_0\in\Omega$ such that 
	\[
	\partial_z^k a_1\!\left(x,\frac{|x|^2}{2}\right)
	\neq
	\partial_z^k a_2\!\left(x,\frac{|x|^2}{2}\right),
	\]
	then by continuity we have
	\[
		\partial_z^k a_1\!\left(x,\frac{|x|^2}{2}\right)
		\neq
		\partial_z^k a_2\!\left(x,\frac{|x|^2}{2}\right)
	\]
	in some ball $B\subset \Omega$ continaining $x_0$. Now, suppose that the boundary value $\psi$ satisfies $\psi\in H^{3/2}(\partial\Omega)$. Since $a_j\in C^{2m+1}(\overline{\Omega}\times \R\times\R^n)$, we have $q\in C^1(\overline{\Omega})$ and thus elliptic regularity theory \cite[Theorem~8.12]{GT} ensures that $v\in H^2(\Omega)$. Hence, we may apply the (weak) unique continuity property of the Schr\"odinger operator $-\Delta+q$ to conclude that $v=0$ in $\Omega$. This is impossible, because we assumed that $\psi\neq 0$. Therefore,
	\[
	\partial_z^k a_1\!\left(x,\frac{|x|^2}{2}\right)
	=
	\partial_z^k a_2\!\left(x,\frac{|x|^2}{2}\right)
	\]
	in $\Omega$. This finishes the proof.
\end{proof}

\section{Inverse source problems: uniqueness and recent developments }
\label{sec: source problems}

An interesting special case of the model studied in the present article is obtained when the nonlinearity depends only on the spatial variable,
\[
a(x,z,p)=f(x).
\]
In this situation the Monge--Amp\`ere equation reduces to
\[
\det D^2u=f(x),
\]
and the coefficients of the linearized operator simplify considerably. Indeed,
if $\Phi$ is a uniformly convex background solution, then
\[
b_\Phi=\nabla_p a(x,\Phi,\nabla\Phi)=0,
\qquad
q_\Phi=\partial_z a(x,\Phi,\nabla\Phi)=0,
\]
and therefore
\[
L_\Phi v
=
\Div(\mathcal C_\Phi\nabla v),
\qquad
\mathcal C_\Phi=\cof D^2\Phi.
\]

Consequently, the reduction principle shows that the nonlinear inverse problem
for
\[
\det D^2u=f(x)
\]
reduces to an anisotropic Calder\'on problem for the conductivity tensor
\[
\mathcal C_\Phi=\cof D^2\Phi.
\]
This reduction gives an exact conclusion within the present cofactor-Hessian
class in dimension two and, in dimensions $n\ge3$, under the standard
real-analytic hypotheses.

\begin{corollary}[Inverse source uniqueness in planar and analytic settings]
	\label{cor: source uniqueness planar analytic}
	Let $0<\alpha<1$, let $\Omega\subset\mathbb R^n$ be a bounded domain
	with $C^{2,\alpha}$ boundary, and let
	\[
	f_j\in C^{0,\alpha}(\overline\Omega),
	\qquad f_j>0,
	\qquad j=1,2.
	\]
	Suppose that there are $\phi\in C^{2,\alpha}(\partial\Omega)$ and uniformly
	uniformly convex functions $\Phi_j\in C^{2,\alpha}(\overline\Omega)$ satisfying
	\[
	\begin{cases}
	\det D^2\Phi_j=f_j&\text{in }\Omega,\\
	\Phi_j=\phi&\text{on }\partial\Omega,
	\end{cases}
	\qquad j=1,2.
	\]
	Assume the boundary compatibilities
	\[
	\partial_\nu\Phi_1=\partial_\nu\Phi_2
	\quad\text{and}\quad
	f_1=f_2
	\qquad\text{on }\partial\Omega,
	\]
	and equality of the local nonlinear Cauchy data sets
	\[
	\mathscr C_{f_1}^{(\phi,\partial_\nu\Phi_1|_{\partial\Omega})}
	=
	\mathscr C_{f_2}^{(\phi,\partial_\nu\Phi_2|_{\partial\Omega})}.
	\]
	Suppose, in addition, that either
	\begin{enumerate}[(i)]
		\item $n=2$ and $\Omega$ is simply connected; or
		\item $n\ge3$ and $\Omega$ and the principal coefficients
		$\mathcal C_{\Phi_j}=\cof D^2\Phi_j$ satisfy the real-analytic hypotheses
		of \cite{lee1989determining}.
	\end{enumerate}
	Then
	\[
	f_1=f_2\quad\text{in }\Omega.
	\]
\end{corollary}


\begin{remark}[Existence of source backgrounds]
	The existence assumption on the backgrounds $\Phi_j$ in
	Corollary~\ref{cor: source uniqueness planar analytic} follows from the
	global regularity theory for the Monge--Amp\`ere
	Dirichlet problem. More precisely, the result used here is
	\cite[Theorem~4.1]
	{trudinger2008monge}; see also
	\cite[Corollary~2.11 and Theorem~2.14]{Alessio}
	for existence and uniqueness in the Alexandrov class.

	Suppose that $\Omega$ is uniformly convex with
	$\partial\Omega\in C^3$, that $\phi$ is the boundary trace of a
	function in $C^3(\overline\Omega)$, and that
	\[
	f_j\in C^{0,\alpha}(\overline\Omega),
	\qquad
	f_j\geq\lambda>0,
	\qquad j=1,2.
	\]
	Hence
	\cite[Theorem~4.1]{trudinger2008monge} (see also \cite[Theorem~1.1]{trudingerwang2008boundary})
	gives a unique convex solution
	$\Phi_j\in C^{2,\alpha}(\overline\Omega)$ of
	\[
	\begin{cases}
	\det D^2\Phi_j=f_j&\text{in }\Omega,\\
	\Phi_j=\phi&\text{on }\partial\Omega.
	\end{cases}
	\]
	Because $D^2\Phi_j$ is continuous and nonnegative definite and
	\[
	\det D^2\Phi_j=f_j\geq\lambda,
	\]
	the Hessian is positive definite at every point of
	$\overline\Omega$. By compactness, its smallest eigenvalue has a
	positive lower bound. Thus $\Phi_j$ is uniformly convex and
	satisfies the background assumptions of
	Corollary~\ref{cor: source uniqueness planar analytic}.

    Besides these positive results, we recall the observation of  \cite{wang1996regularity} that $\Phi_j$ may fail to be $C^2$, or even $W^{2,p}$ for large $p$, when the Dirichlet datum $\phi$ or the boundary $\partial\Omega$ are only $C^{2,1}$ regular.
\end{remark}

\begin{remark}[Pure-source regularity]
\label{rem: pure-source regularity}
	The $C^{0,\alpha}$ assumption on the sources $f_j$ in
	Corollary~\ref{cor: source uniqueness planar analytic} is sufficient even
	though Theorem~\ref{thm: local linearization MA}, for a general
	nonlinearity $a(x,z,p)$ with $m=0$, is stated under the stronger condition
	$a\in C^{1,\alpha}$. Indeed, once a uniformly convex background
	$\Phi\in C^{2,\alpha}(\overline\Omega)$ satisfying
	$\det D^2\Phi=f$ is given, the nonlinear map in the pure-source problem is
	\[
	(\eta,u)\mapsto
	\bigl(\det D^2u-f,\,u|_{\partial\Omega}-\eta\bigr).
	\]
	This map is smooth into
	$C^{0,\alpha}(\overline\Omega)\times C^{2,\alpha}(\partial\Omega)$ when
	$f\in C^{0,\alpha}(\overline\Omega)$ is fixed. Moreover, the derivatives of
	$f$ with respect to $z$ and $p$, and hence all lower-order coefficients in
	the linearized operator, vanish. The same implicit-function argument
	therefore gives the nearby solutions and their first variation without
	requiring one additional derivative of $f$.
\end{remark}

\begin{proof}[Proof of Corollary~\ref{cor: source uniqueness planar analytic}]
	For the source equation, the lower-order coefficients vanish:
	\[
	b_{\Phi_j}^j=q_{\Phi_j}^j=0.
	\]
    Thus, we are in the pure conductivity case for the linearized inverse problems. Moreover, since either \ref{two-dim case} or \ref{higher-dim case} of Theorem~\ref{thm: anisotropic uniqueness gauge} are satisfied, we may conclude from Theorem~\ref{thm: anisotropic uniqueness gauge} (with $m=0$) that
	\begin{equation}
    \label{eq: equal cofactor inverse source}
	\mathcal C_{\Phi_1}=\mathcal C_{\Phi_2}.
	\end{equation}
	Recall that for every positive definite matrix $A$,
	 \[
    \cof A=(\det A)A^{-1}.
    \]
    Thus, taking the determinant shows that
    \[
	A=(\det(\cof A))^{1/(n-1)}(\cof A)^{-1},
	\] 
	and hence the cofactor matrix determines $A$ uniquely. Consequently, \eqref{eq: equal cofactor inverse source} yields
	\[
	D^2\Phi_1=D^2\Phi_2.
	\]
    Therefore, the related Monge--Ampère equations ensure that
	\[
	f_1=\det D^2\Phi_1=\det D^2\Phi_2=f_2\quad\text{in }\Omega.
	\]
\end{proof}

More generally, in any dimension in which the associated anisotropic inverse
problem uniquely determines the cofactor matrix $\mathcal C_\Phi$, one can
argue as in the proof of
Corollary~\ref{cor: source uniqueness planar analytic} to obtain
\[
D^2\Phi_1=D^2\Phi_2.
\]
Consequently,
\[
f_1
=\det D^2\Phi_1
= \det D^2\Phi_2
= f_2.
\]
Hence inverse source problems for the Monge--Amp\`ere
equation are closely connected to rigidity properties of the class
\[
\mathfrak C
=
\{\cof D^2\Phi;\Phi \text{ uniformly convex}\},
\]
which forms a distinguished subclass of anisotropic conductivities satisfying
the Piola identity
\[
\Div(\cof D^2\Phi)=0.
\]
Remark~\ref{rem: gauge rigidity} explains that the usual boundary-fixing
diffeomorphism gauge of the anisotropic Calder\'on problem cannot occur nontrivially within this divergence-free
class: if $\mathcal C_2=F_*\mathcal C_1$ and both tensors belong to
$\mathfrak C$, then $F=\mathrm{Id}$. 

The relevant cofactor-Hessian class $\mathfrak C$ depends on the regularity of the backgrounds. In dimension two, the argument applies at the $L^\infty$ conductivity level, whereas in dimensions $n\ge3$ the result considered here applies to the real-analytic subclass, since we invoke the real-analytic anisotropic theorem of Lee and Uhlmann.

According to the previous observation, the remaining question in
dimensions $n\ge 3$, for general $C^{k,\alpha}$ or rougher background solutions $\Phi$, is whether
equality of the boundary measurements implies gauge equivalence for tensors in the cofactor-Hessian class $\mathfrak C$.
This is a genuinely open anisotropic uniqueness problem and is reserved
for future work.

\subsection{Comparison with recent inverse source results}

During the preparation of the present manuscript, two closely related works
appeared independently. 

\subsubsection{Small-data approach to the planar inverse source problem}

Liimatainen and Lin
\cite{liimatainen2025inverseproblemmongeampereequation} study the source equation
\begin{equation}
\label{eq: source eq}
\det D^2u=f(x)
\end{equation}
on a bounded uniformly convex planar domain $\Omega\subset\R^2$ in the smooth setting. They use
the nonlinear Dirichlet-to-Neumann map for sufficiently small boundary values,
\[
\Lambda_f(\varphi)=\partial_\nu u_\varphi|_{\partial\Omega},
\]
where $u_\varphi$ is the solution with boundary value $\varphi$, and assume that the two sources agree through second order on the boundary.
The first linearization is viewed as the nondivergence-form equation
\begin{equation}
    \label{eq: linearized MA eq}\bigl((D^2\Phi)^{-1}\bigr)_{ij}\partial_{ij}v=0,
\end{equation}
where $\Phi$ is the unique solution to \eqref{eq: source eq} with homogeneous boundary conditions.

Liimatainen and Lin organize their proof in two stages. The first uses the
metric formulation of \eqref{eq: linearized MA eq} in nondivergence form to
determine the background metric $g = D^2\Phi$ up to a boundary-fixing coordinate change $F$ and
a positive conformal factor $c>0$; the special first-order structure is then used to
remove the coordinate ambiguity $F=\rm Id$. The second stage seeks to determine the
remaining conformal factor $c$ from the second linearization, using isothermal
coordinates, refined complex geometrical optics solutions, and unique
continuation for a nonlocal $\overline\partial$-equation.

For the weighted conormal data, $\partial_{\nu_g}v|_{\partial\Omega}$, used in this metric formulation, conformal
rescaling is an intrinsic invariance in dimension two. Thus the boundary
normalization of the conformal factor must come from additional boundary
information rather than from the linearized map alone. In the Monge--Amp\`ere
source application, the boundary determination in
\cite{liimatainen2025inverseproblemmongeampereequation} is designed to provide
precisely such information: once the relevant boundary jets of the background
metrics are incorporated, they give the compatibility needed in the cited
two-dimensional metric theorem and normalize the conformal factor after the
boundary-fixing coordinate change has been removed. This source-specific
boundary information is absent from the unrestricted metric problem.

\medskip
\noindent\emph{Relation with the cofactor-conductivity formulation.}
\par\smallskip
\noindent
The measurements of Liimatainen and Lin directly contain the local data
required by Corollary~\ref{cor: source uniqueness planar analytic}. Indeed,
let $B_\delta(0)$ be a common ball on which
\[
\Lambda_{f_1}=\Lambda_{f_2},
\]
choose any $\varphi\in B_\delta(0)$, and let
\[
\Phi_j:=u^{(j)}_\varphi,
\qquad j=1,2,
\]
be the corresponding uniformly convex solutions. They have the common
boundary value $\varphi$, and equality of the nonlinear
Dirichlet-to-Neumann maps at $\varphi$ gives
\[
\partial_\nu\Phi_1
=\Lambda_{f_1}(\varphi)
=\Lambda_{f_2}(\varphi)
=\partial_\nu\Phi_2
\qquad\text{on }\partial\Omega.
\]
After choosing $\varepsilon>0$ such that
$B_\varepsilon(\varphi)\subset B_\delta(0)$, equality of the two maps on this
smaller ball gives equality of the local nonlinear Cauchy data sets around
the backgrounds $\Phi_j$. Moreover, their second-order boundary matching
assumption implies the zeroth-order compatibility
\[
f_1=f_2\qquad\text{on }\partial\Omega.
\]
Thus all compatibility and measurement assumptions of
Corollary~\ref{cor: source uniqueness planar analytic} are satisfied.

For completeness, we recall the mechanism behind the corollary. Set
\[
\mathcal C_j=\cof D^2\Phi_j.
\]
Multiplying the first linearized equation by
$f_j=\det D^2\Phi_j$ and using the Piola identity $\Div\mathcal C_j=0$ gives
\[
\Div(\mathcal C_j\nabla v_j)=0.
\]
The common boundary value, equality of the background normal derivatives,
and $f_1=f_2$ on $\partial\Omega$ imply, by
Remark~\ref{rem: boundary determination cofactor}, that
\[
D^2\Phi_1=D^2\Phi_2
\quad\text{and}\quad
\mathcal C_1=\mathcal C_2
\qquad\text{on }\partial\Omega.
\]
The linearized normal data therefore determine the conormal
Dirichlet-to-Neumann maps for $\mathcal C_1$ and $\mathcal C_2$. Planar
anisotropic conductivity uniqueness gives
\[
\mathcal C_2=F_*\mathcal C_1,
\qquad
F|_{\partial\Omega}=\mathrm{Id}.
\]
Since both tensors are divergence free, the coordinate-function argument in
Remark~\ref{rem: gauge rigidity} forces $F=\mathrm{Id}$. Hence
$\mathcal C_1=\mathcal C_2$. In two dimensions the cofactor map is
involutive, $\cof(\cof A)=A$, and therefore
\[
D^2\Phi_1=D^2\Phi_2\quad\text{in }\Omega.
\]
Thus the Monge--Ampère equation gives
\[
    f_1=f_2\quad\text{in }\Omega.
\]

Consequently, under the smoothness, convexity, and forward well-posedness
assumptions of Liimatainen and Lin, their planar source theorem falls within
the scope of Corollary~\ref{cor: source uniqueness planar analytic}. This
gives a complementary proof through the first-linearization
cofactor-conductivity formulation. For this route, their second-order
boundary matching can be weakened to $f_1=f_2$ on $\partial\Omega$.

In the locally well-posed regime, a nonlinear Cauchy data set is the graph of
the corresponding nonlinear Dirichlet-to-Neumann map, so the two formulations
carry the same information when considered on the same neighborhood. The
Liimatainen--Lin map on the full ball $B_\delta(0)$ contains, in particular,
the single-background local measurements required by our corollary around
every $\varphi$ lying strictly inside that ball. A further distinction is
regularity: their theorem is formulated for smooth sources, whereas
Corollary~\ref{cor: source uniqueness planar analytic} works at finite H\"older
regularity, with $f_j\in C^{0,\alpha}(\overline\Omega)$ and
$\Phi_j\in C^{2,\alpha}(\overline\Omega)$. This difference reflects the
source-specific observation in Remark~\ref{rem: pure-source regularity}; it
does not affect the common smooth regime of the two approaches.

\subsubsection{Large-data approach in dimensions \texorpdfstring{$n\geq 2$}{high}}

Cârstea and Ghosh \cite{cârstea2026inversesourceproblemmongeampere} consider the source equation
in every dimension $n\ge2$, for smooth positive sources on a smooth uniformly
convex domain. They prove uniqueness from equality of the full nonlinear
Dirichlet-to-Neumann maps and, more precisely, use a special family of large
cylindrical boundary values. A large-data asymptotic expansion extracts chord
integrals of the source, after which injectivity of the Euclidean X-ray
transform yields uniqueness. Their theorem does not require the sources to
agree on the boundary. Thus, for the smooth source equation, it gives
uniqueness in dimensions $n\ge3$ without the real-analytic and boundary
compatibility assumptions appearing in
Corollary~\ref{cor: source uniqueness planar analytic}. Their argument,
however, uses global large data, whereas the present framework concerns local
data around a fixed background and applies to the general nonlinearity
$a(x,u,\nabla u)$.

As in the small-data setting above, equality of the global nonlinear
Dirichlet-to-Neumann maps contains the single-background local measurements
used in Corollary~\ref{cor: source uniqueness planar analytic}. Consequently,
if one additionally assumes
\[
f_1=f_2\qquad\text{on }\partial\Omega,
\]
then the global C\^arstea--Ghosh measurements satisfy the data and boundary
compatibility assumptions of Corollary~\ref{cor: source uniqueness planar analytic}.
Thus our corollary also yields uniqueness from this restricted part of their
measurements in the planar simply connected case, and in higher dimensions
under the real-analytic anisotropic hypotheses stated there. This observation
does not subsume their result: their large-data argument applies in every
dimension without assuming that the sources agree on the boundary.

A further distinction concerns regularity. Cârstea and Ghosh formulate their
large-boundary-data uniqueness theorem in the smooth classical regime, with
positive sources $f_j\in C^\infty(\overline\Omega)$, and explicitly note that
a lower-regularity formulation would require additional justification of the
boundary traces. Smoothness also enters their proof substantively. Their
chordwise correction $w_\omega^f$ must have bounded full second derivatives,
including $D_y^2w_\omega^f$ and $D_y\partial_s w_\omega^f$, since these occur
in the determinant defect
\[
E_\omega^f
=f\,\operatorname{tr}(D_y^2w_\omega^f)
-|D_y\partial_s w_\omega^f|^2
\]
and in the global strict-convexity and comparison-barrier arguments. Their
estimate for $D^2w_\omega^f$ is controlled by the $C^2$ norm of $f$, while
the glancing construction provides a global smooth extension that keeps these
second derivatives bounded up to the glancing set.

The $C^\infty$ hypothesis is plausibly stronger than necessary: careful
finite-regularity bookkeeping may extend their method to
$f\in C^{2,\alpha}$, perhaps even to $C^2$. This is an observation about the
proof, however, and not a theorem established in their paper. Their argument
does not directly extend to merely $f\in C^{0,\alpha}$. At that regularity the
transverse second derivatives of $w_\omega^f$ need not exist as bounded
functions, so the pointwise determinant expansion, the strict-convexity
argument, and the determinant inequalities for the barriers are unavailable.
Mollification is not an automatic remedy: the required derivative norms may
blow up, and one would need uniform large-$t$ asymptotics together with
boundary-normal stability at the $t^{1-n}$ scale.

By contrast, for a positive source $f\in C^{0,\alpha}$ in the pure-source
regularity regime of Remark~\ref{rem: pure-source regularity}, the background satisfies
$\Phi\in C^{2,\alpha}$ and hence
\[
\mathcal C_\Phi=\cof D^2\Phi\in C^{0,\alpha}.
\]
The Piola identity remains valid distributionally, and the first
linearization can be treated as a divergence-form conductivity equation,
without differentiating $f$ twice or constructing a globally $C^2$
chordwise correction. Thus, although the smooth large-data method of Cârstea
and Ghosh may plausibly admit a finite-regularity extension to
$C^{2,\alpha}$ (perhaps $C^2$), it does not directly reach
$C^{0,\alpha}$; at that H\"older regularity, the present first-linearization
cofactor-conductivity formulation remains available without the globally
$C^2$ chordwise corrections required by their determinant barriers.

\subsubsection{Comparison conclusion}

The measurement assumptions are therefore nested rather than complementary:
the Dirichlet-to-Neumann map data in both recent source works contain the
single-background local measurements used here. The substantive distinctions
instead concern scope, auxiliary assumptions, regularity, and method.
Liimatainen and Lin treat the smooth planar source problem, with boundary
matching, through a metric and second-linearization argument using
complex-analytic tools. Cârstea and Ghosh treat smooth sources in every
dimension, without boundary matching, by global large-data asymptotics and
integral geometry. The present work uses local data at finite H\"older
regularity, develops the cofactor-conductivity reduction, and applies more
generally to nonlinearities depending on $(x,u,\nabla u)$, together with a
systematic higher-order recovery framework. In
particular, the developed reduction principle relates the nonlinear inverse problem to
anisotropic Calder\'on-type problems for the cofactor matrix
\[
\mathcal C_\Phi=\cof D^2\Phi,
\]
while the higher-order linearization identities provide access to the
quantities
\[
\nabla_p a(x,\Phi,\nabla\Phi),
\qquad
\partial_z a(x,\Phi,\nabla\Phi),
\]
and, more generally, to the higher-order derivatives of $a$ along the
background jet
\[
(x,\Phi(x),\nabla\Phi(x)).
\]

\subsubsection{Outlook toward Alexandrov backgrounds and optimal transport}

The cofactor formulation suggests a possible extension of the planar inverse problem to backgrounds understood in the Alexandrov sense. Such weak backgrounds arise naturally in optimal transportation: for the quadratic cost, the Brenier potential $\Phi$ is convex and, formally, satisfies
\[
\rho_1(\nabla\Phi(x))\det D^2\Phi(x)=\rho_0(x),
\]
with the corresponding equation interpreted in the Alexandrov or Brenier sense at low regularity; see \cite{Alessio-Guido}. Thus the optimal-transport equation belongs to the gradient-dependent Monge--Ampère class considered in the present work, although its natural second boundary condition differs from the Dirichlet boundary measurements studied here.

Modern Monge--Ampère regularity theory provides quantitative second-order information for weak solutions, ranging from the interior $W^{2,1+\varepsilon}$ theory for uniformly convex solutions with density bounded above and below to recent interior $W^{2,\infty}$ estimates in the uniformly convex constant-density setting; see, for example, \cite{DePhilippisFigalliSavin,ChenZhou}. The $W^{2,\infty}$ threshold is particularly natural for the planar inverse problem. Indeed, as observed in \cite[Remark~1.1]{Alessio}, if
\[
\Phi\in W^{2,\infty}(\Omega),
\qquad
\det D^2\Phi\geq a_0>0,
\]
then the $L^\infty $ upper bound for $D^2\Phi $ also gives a quantitative positive lower bound for its eigenvalues. Consequently,
\[
\mathcal C_\Phi=\cof D^2\Phi
\]
is a bounded uniformly elliptic conductivity. The planar anisotropic Calderón theorem is available at this measurable-coefficient level, while the Piola identity and the resulting gauge-rigidity argument continue to make sense distributionally. Thus a suitable $W^{2,\infty}$ theory for Alexandrov backgrounds could place the first linearization directly within the low-regularity planar conductivity framework and may eventually provide a link between inverse boundary problems and low-regularity optimal transport.

Such an extension is not established here. The strongest recent Hessian estimates are principally interior and, in the $W^{2,\infty}$ setting just mentioned, concern the constant-density equation, whereas an inverse boundary problem also requires sufficient regularity near $\partial\Omega$, differentiability of the nonlinear solution and boundary maps in a weak solution class, and a rigorous identification of the linearized measurements with the weak conormal map. In the optimal-transport setting, one would additionally have to adapt the measurement formulation to the natural transport boundary condition. A possible intermediate problem would therefore allow rough densities in the interior while retaining enough regularity in a collar of the boundary to recover the required boundary compatibility.

	\medskip

	\appendix

	\section{Examples and structure of nonlinearities}
	\label{sec: appendix examples}

	In this section we present some example nonlinearities.

	\begin{remark}[General construction of admissible nonlinearities]
		\label{rem: construction nonlinearities}
		A convenient way to construct nonlinearities satisfying
		\[
		a(x,\Phi(x),\nabla \Phi(x)) = 1
		\]
		is to parameterize deviations from the background by
		\[
		z - \Phi(x),
		\qquad
		p - \nabla \Phi(x).
		\]
		More precisely, let
		\[
		F : \Omega \times \R \times \R^n \to \R
		\]
		be a smooth function satisfying
		\[
		F(x,0,0) = 0.
		\]
		Then the function
		\[
		a(x,z,p)
		=
		1 + F\!\left(x,\, z - \tfrac{|x|^2}{2},\, p - x \right)
		\]
		satisfies
		\[
		a(x,\Phi(x),\nabla \Phi(x)) = 1
		\quad \text{for all } x \in \Omega.
		\]
		This representation shows that the admissible class of nonlinearities is large,
		and that the inverse problem naturally concerns the recovery of the coefficients
		appearing in the expansion of $F$.
	\end{remark}

	\begin{example}[Zeroth-order nonlinearities]
		\label{ex: zeroth order}
		Let $c \in C^{m+1,\alpha}(\overline{\Omega})$ and $\Theta \in C_b^1(\R)$ be given.
		Then
		\[
		a(x,z)
		=
		1
		+
		c(x)\Big(\Theta(z) - \Theta\!\big(\tfrac{|x|^2}{2}\big)\Big)
		\]
		satisfies
		\[
		a(x,\Phi(x)) = 1.
		\]

		Typical examples include:
		\begin{align*}
		a(x,z)
		&= 1 + c(x)\Big(z - \tfrac{|x|^2}{2}\Big), \\
		a(x,z)
		&= 1 + c(x)\left(z^2 - \tfrac{|x|^4}{4}\right), \\
		a(x,z)
		&= 1 + c(x)\sin\!\Big(z - \tfrac{|x|^2}{2}\Big), \\
		a(x,z)
		&= \exp\!\Big(c(x)\big(z - \tfrac{|x|^2}{2}\big)\Big).
		\end{align*}

		In all cases, the background condition is satisfied, and
		\[
		\partial_z a(x,\Phi(x)) = c(x)\Theta'(\Phi(x))
		\]
		is the quantity recovered in the linearized inverse problem.
	\end{example}

	\begin{example}[Zeroth and first-order nonlinearities]
		\label{ex: first order}
		Let $c \in C^{m+1,\alpha}(\overline{\Omega})$, $d \in C^{m+1,\alpha}(\overline{\Omega};\R^n)$,
		and let $\Theta \in C_b^1(\R)$, $\Psi \in C_b^1(\R^n)$ be given.
		Then
		\[
		a(x,z,p)
		=
		1
		+
		c(x)\Big(\Theta(z) - \Theta\!\big(\tfrac{|x|^2}{2}\big)\Big)
		+
		d(x)\cdot\Big(\Psi(p) - \Psi(x)\Big)
		\]
		satisfies
		\[
		a(x,\Phi(x),\nabla \Phi(x)) = 1.
		\]

		Typical examples include:
		\begin{align*}
		a(x,z,p)
		&= 1 + c(x)\Big(z - \tfrac{|x|^2}{2}\Big) + d(x)\cdot (p - x), \\
		a(x,z,p)
		&= 1 + c(x)\Big(z - \tfrac{|x|^2}{2}\Big) + d(x)\big(|p|^2 - |x|^2\big), \\
		a(x,z,p)
		&= \exp\!\Big(
		c(x)\big(z - \tfrac{|x|^2}{2}\big)
		+
		d(x)\cdot (p - x)
		\Big).
		\end{align*}

		These examples illustrate how both zeroth- and first-order coefficients enter
		the linearized operator through
		\[
		q(x) = \partial_z a(x,\Phi(x),\nabla\Phi(x)),
		\qquad
		b(x) = \nabla_p a(x,\Phi(x),\nabla\Phi(x)).
		\]
	\end{example}

	\begin{remark}[Interpretation]
		\label{rem: interpretation}
		The normalization
		\[
		a(x,\Phi(x),\nabla \Phi(x)) = 1
		\]
		should be viewed as fixing a reference solution $\Phi$ and describing the
		nonlinearity relative to this background. The inverse problem then aims at
		recovering the coefficients governing the deviations
		\[
		z - \Phi(x), \qquad p - \nabla\Phi(x),
		\]
		from boundary measurements. In particular, the reconstruction results obtained
		in the previous sections determine the derivatives of $a$ along the trajectory
		\[
		x \mapsto (x,\Phi(x),\nabla \Phi(x)).
		\]
	\end{remark}

	\begin{remark}[Normalization with respect to a quadratic background]
		\label{rem: normalization quadratic background}
		In the quadratic background case
		\[
		\Phi(x) = \tfrac{1}{2}|x|^2,
		\]
		we have
		\[
		D^2\Phi = \id,
		\qquad
		\det D^2\Phi = 1,
		\qquad
		\nabla \Phi(x) = x.
		\]
		Hence, the condition that $\Phi$ is a common background solution for $a_1$ and $a_2$
		amounts to the normalization
		\[
		a_j\!\left(x,\tfrac{|x|^2}{2},x\right) = 1
		\qquad \text{for all } x \in \Omega, \quad j=1,2.
		\]
		This condition is not restrictive: it simply expresses that we study nonlinearities
		in a neighborhood of a fixed known solution $\Phi$.
	\end{remark}

  \medskip 

		\noindent\textbf{Acknowledgment.} 

		\begin{itemize}
			\item G.~Uhlmann is partially supported by NSF.
			\item P.~Zimmermann is supported by the Swiss National Science Foundation (SNSF) under Grant Nos.~214500 and 239130.
		\end{itemize}

  \medskip

    \bibliography{refs} 

    \bibliographystyle{alpha}

\end{document}